\documentclass[11pt]{amsart}
\usepackage{amssymb}
\usepackage{latexsym,bm}
\usepackage{amsmath,amsfonts,amssymb}
\usepackage{amscd}

\newtheorem{theorem}{Theorem}[section]
\newtheorem{lemma}[theorem]{Lemma}
\newtheorem{corollary}[theorem]{Corollary}

\theoremstyle{definition}
\newtheorem{definition}[theorem]{Definition}
\newtheorem{example}[theorem]{Example}

\theoremstyle{remark}

\numberwithin{equation}{section}

\DeclareMathOperator{\cha}{char} \DeclareMathOperator{\End}{End}
\DeclareMathOperator{\Hom}{Hom} \DeclareMathOperator{\id}{Id}
\DeclareMathOperator{\chf}{ch} \DeclareMathOperator{\Tr}{Tr}
\DeclareMathOperator{\Mod}{Mod}
\DeclareMathOperator{\tr}{tr}
\DeclareMathOperator{\Res}{Res}
\newcommand{\HH}{\mathcal H}
\newcommand{\Z}{\mathbb Z}
\newcommand{\BS}{\mathfrak S}
\newcommand{\DD}{\mathfrak D}
\newcommand{\ft}{\mathfrak t}
\newcommand{\lam}{\lambda}
\newcommand{\PP}{\mathcal{P}}
\newcommand{\ulam}{\underline{\lambda}}
\newcommand{\umu}{\underline{\mu}}

\begin{document}

\title[\null]{On the decomposition numbers of the Hecke algebra of type $D_n$ when $n$ is even}

\author{Jun Hu}

\address{Department of Applied Mathematics\\
School of Science\\
Beijing Institute of Technology\\
Beijing 100081 P.R. China}
\email{junhu303@yahoo.com.cn}

\subjclass[2000]{Primary 17B37, 20G42; Secondary 20G15}
\date{}
\keywords{Iwahori--Hecke algebra; $q$-Schur algebra; dual Specht modules}

\begin{abstract}
Let $n\geq 4$ be an even integer. Let $K$ be a field with $\cha
K\neq 2$ and $q$ an invertible element in $K$ such that
$\prod_{i=1}^{n-1}(1+q^i)\neq 0$. In this paper, we study the
decomposition numbers over $K$ of the Iwahori--Hecke algebra
$\HH_q(D_n)$ of type $D_n$. We obtain some equalities which relate
its decomposition numbers with certain Schur elements and the decomposition numbers of various
Iwahori--Hecke algebras of type $A$ with the same parameter $q$. When $\cha K=0$, this completely determine
all of its decomposition numbers. The main tools we used are the
Morita equivalence theorem established in \cite{Hu1} and certain
twining character formulae of Weyl modules over a tensor product of
two $q$-Schur algebras.
\end{abstract}
\maketitle

\section{Introduction}\smallskip

Let $n$ be a natural number. Let $K$ be a field and $q, Q$ two
invertible elements in $K$. Let $W_n$ be the Weyl group of type
$A_{n-1}$ or of type $B_n$. Let $\HH(W_n)$ be the Iwahori--Hecke
algebra of $W_n$ with parameter $q$ if $W_n=W(A_{n-1})$; or with
parameters $q, Q$ if $W_n=W(B_n)$. The modular representation theory
of $\HH(W_n)$ has been well studied in the papers \cite{AJ},
\cite{DJ1}, \cite{DJ2}, \cite{DJ3}, \cite{DJMu} and \cite{Mu}. In
fact, most of the results of the modular representation theory of these
algebras have been generalized to a more general class of
algebras---the cyclotomic Hecke algebras of type $G(r,1,n)$, where
$r\in\mathbb{N}$. The latter was now fairly well understood by the
work of \cite{A1}, \cite{A2}, \cite{AM}, \cite{DJM} and \cite{DM}.\smallskip

This paper is concerned with the Iwahori--Hecke algebra $\HH_q(D_n)$
of type $D_n$. By definition, $\HH_{q}(D_n)$ is the associative
unital $K$-algebra with generators $T_u, T_1,\cdots,$ $T_{n-1}$
subject to the following relations
$$\begin{aligned}
&(T_u+1)(T_u-q)=0,\\
&(T_i+1)(T_i-q)=0,\quad\text{for $1\leq i\leq n-1$,}\\
&T_uT_{2}T_u=T_{2}T_{u}T_{2},\quad T_uT_1=T_1T_u,\\
&T_iT_{i+1}T_i=T_{i+1}T_{i}T_{i+1},\quad\text{for $1\leq i\leq n-2$,}\\
&T_iT_j=T_jT_i,\quad\text{for $1\leq i<j-1\leq n-2$,}\\
&T_uT_i=T_iT_u,\quad\text{for $2<i<n$.}\end{aligned}
$$
The algebra $\HH_{q}(D_n)$ can be embedded into a Hecke algebra
$\HH_q(B_n)$ of type $B_n$ with parameters $\{q,1\}$ as a ``normal" subalgebra. Namely, let
$\HH_{q}(B_n)$ be the associative unital $K$-algebra with generators
$T_0, T_1,\cdots, T_{n-1}$ subject to the following relations
$$\begin{aligned}
&(T_0+1)(T_0-1)=0,\\
&T_0T_1T_0T_1=T_1T_0T_1T_0,\\
&(T_i+1)(T_i-q)=0,\quad\text{for $1\leq i\leq n-1$,}\\
&T_iT_{i+1}T_i=T_{i+1}T_{i}T_{i+1},\quad\text{for $1\leq i\leq n-2$,}\\
&T_iT_j=T_jT_i,\quad\text{for $0\leq i<j-1\leq n-2$.}\end{aligned}
$$
Then the map $\iota$ which sends $T_u$ to $T_0T_1T_0$, and $T_i$ to
$T_i$ (for each integer $i$ with $1\leq i\leq n-1$) can be uniquely
extended to an injection of $K$-algebras. Throughout this paper, we shall
always identify $\HH_{q}(D_n)$ with the subalgebra
$\iota(\HH_{q}(D_n))$ using this embedding $\iota$.\smallskip

{\it Henceforth, we shall assume that the characteristic of the field $K$ (denoted by $\cha K$) is not equal $2$}. In
this case,
if $q$ is not a root of unity, then $\HH_q(D_n)$ is semisimple. Since we are only interested in the modular (i.e.,
non-semisimple) case,
we shall also assume that $q$ is a root of unity in $K$.
Let $e$ be the smallest positive integer such that
$1+q+q^2+\cdots+q^{e-1}=0$. If $q=1$, then $e=\cha K$. The modular representation
theory of $\HH_{q}(D_n)$ over $K$ was studied in
a number of papers \cite{G}, \cite{Hu1}, \cite{Hu2}, \cite{Hu4},
\cite{Ja1}, \cite{P1}. The algebra $\HH_{q}(D_n)$ is a special case of a more general class of algebra---the
cyclotomic Hecke
algebras of type $G(r,p,n)$. The latter was studied in \cite{Hu3}, \cite{Hu5}, \cite{Hu6}, \cite{Hu7}, \cite{HMa} and
\cite{HSh}.
The papers \cite{Hu1}, \cite{Hu2}
and \cite{P1} studied the restriction to $\HH_{q}(D_n)$ of simple
$\HH_q(B_n)$-modules using the combinatorics of Kleshchev
bipartitions; while the papers \cite{G} and \cite{Ja1} studied the
simple $\HH_{q}(D_n)$-modules with the aim of constructing the
so-called ``canonical basic set". In both approaches, simple $\HH_{q}(D_n)$-modules have been classified but using different parameterizations.
\smallskip

One of the major open problems in the modular representation theory of Hecke algebras is the determination of their
decomposition numbers.
In the case of type $A$ and type $B$ (or more generally, of type $G(r,1,n)$), thanks to the work of \cite{A1} and
\cite{LLT}, the decomposition numbers can be computed by the evaluation at $1$ of some alternating sum of certain
parabolic affine Kazhdan--Lusztig polynomials when $\cha K=0$ and $q\neq 1$. It is natural to ask what will happen
in the type $D$ case. In \cite{P1}, it was proved that if $n$ is odd and $f_{n}(q):=\prod_{i=1}^{n-1}(1+q^i)\neq 0$
in $K$, then
$$\HH_{q}(D_n)\overset{Morita}{\sim}\bigoplus_{a=(n+1)/2}^{n}\HH_q(\BS_{(a,n-a)}), $$
where $\BS_{(a,n-a)}:=\BS_{\{1,\cdots,a\}}\times\BS_{\{a+1,\cdots,n\}}$ is the parabolic subgroup of the symmetric group
$\BS_n$ on $\{1,2,\cdots,n\}$, $\HH_q(\BS_{(a,n-a)})$ is the parabolic subalgebra of $\HH_q(\BS_n)$ corresponding to
$\BS_{(a,n-a)}$. In this case, if $S_{\ulam}$ and $D_{\umu}$ denote the dual Specht module
and simple module of $\HH_q(B_n)$ labelled by the bipartition
$\ulam=(\lam^{(1)},\lam^{(2)})$ and the $e$-restricted bipartition
$\umu=(\mu^{(1)},\mu^{(2)})$ respectively, then we have the
following equality of decomposition numbers:
$$
\bigl[S_{\ulam}\downarrow_{\HH_q(D_n)}:D_{\umu}\downarrow_{\HH_q(D_n)}\bigr]=
\bigl[S_{\lam^{(1)}}:D_{\mu^{(1)}}\bigr]\bigl[\widehat{S}_{\lam^{(2)}}:\widehat{D}_{\mu^{(2)}}\bigr],
$$
where $S_{\lam^{(1)}}, D_{\mu^{(1)}}$ (resp.,
$\widehat{S}_{\lam^{(2)}}, \widehat{D}_{\mu^{(2)}}$) denote the dual
Specht module and simple module of $\HH_q(\BS_a)$ (resp., of
$\HH_q(\BS_{\{a+1,\cdots,n\}})$). Hence computing the decomposition
numbers of $\HH_q(D_n)$ can be reduced to computing the decomposition
numbers of $\HH_q(\BS_{(a,n-a)})$ where $(n+1)/2\leq a\leq n$.
\smallskip

In \cite{Hu1}, it was proved that if $n$ is even and $f_{n}(q)\neq 0$ in $K$, then \begin{equation}
\label{MoritaEquivalence}
\HH_{q}(D_n)\overset{Morita}{\sim}A(n/2)\oplus\bigoplus_{a=n/2+1}^{n}\HH_q(\BS_{(a,n-a)}), \end{equation}
where $A(n/2)$ is the $K$-subalgebra of $\HH_q(\BS_n)$ generated by $\HH_q(\BS_{(n/2,n/2)})$ and
an invertible element $h(n/2)\in\HH_q(\BS_n)$ (see \cite[Definition 1.5]{Hu1} for definition of $h(n/2)$).
Therefore, in this case, computing the decomposition numbers of $\HH_q(D_n)$ can be reduced to computing the
decomposition numbers of $\HH_q(\BS_{(a,n-a)})$ where $n/2+1\leq a\leq n$ and the decomposition numbers
of the algebra $A(n/2)$. \smallskip

The purpose of this article is to determine the decomposition
numbers for the algebra $A(n/2)$. The main results of this paper
provide some explicit formulae which determine these decomposition
numbers of $\HH_q(D_n)$ in terms of the decomposition numbers of $\HH_q(\BS_{n/2})$
and certain Schur elements, see Theorem \ref{decnumodd2}, Lemma \ref{decnumeven1} (where it is only assumed $\cha
K\neq 2$), Theorem \ref{decnum0} (in the case where $\cha K=0$) and Theorem \ref{mainthm}. Note that these results are
also valid in the case where $q=1$.
\smallskip

The paper is organized as follows. In Section 2 we shall first
recall some known results about the modular
representation theory of the Hecke algebra of type $D_n$ when $n$ is
even in the separated case. Then we shall state our main theorems. In Section 3 we lift
the construction of the algebra $A(n/2)$ to the level of $q$-Schur
algebras and introduce a certain covering $\widetilde{A}(n/2)$ of
$A(n/2)$. Using Schur functor, we transfer the original problem of computing the
decomposition numbers of $A(n/2)$ to the corresponding problem for
$\widetilde{A}(n/2)$. In Section 4, we explicitly compute the
Laurent polynomial $f_{\ulam}(v)$ introduced in \cite[Lemma
3.2]{Hu1} in terms of the Schur elements of the Hecke algebra
$\HH_q(B_n)$ and $\HH_q(\BS_{n/2})$. In Section 5, by computing the
twining character formula of some Weyl modules of a tensor product
of two $q$-Schur algebras, we determine our
desired decomposition numbers when $K$ is of characteristic $0$.
When $K$ is of odd characteristic, our results only give some
equalities about these decomposition numbers modulo $\cha K$.
\bigskip

\section{Preliminaries}\smallskip

{\it From now on until the end of this paper, we assume that
$n=2m\geq 4$ is an even integer, and} \begin{equation}\label{separa}
2f_n(q):=2\prod_{i=1}^{n-1}(1+q^i)\neq 0.
\end{equation}
{\it We shall refer (\ref{separa}) as separation condition and say that we are in the separated case.}
\smallskip

Let $k$ be a positive integer. A sequence of nonnegative integers
$\lam=(\lam_1,\lam_2,\cdots)$ is said to be a composition of $k$
if $\sum_{i\geq 1}\lam_i=k$. A
composition $\lam=(\lam_1,\lam_2,\cdots)$ of $k$ is said to be a
partition of $k$ (denoted by $\lam\vdash k$) if
$\lam_1\geq\lam_2\geq\cdots$. If $\lam$ is a composition of $k$, then
we write $|\lam|=k$. A bipartition of $n$ is an ordered pair
$\ulam=(\lam^{(1)}, \lam^{(2)})$ of partitions such that
$\lam^{(1)}$ is a partition of $a$ and $\lam^{(2)}$ is a partition
of $n-a$ for some integer $a$ with $0\leq a\leq n$. In this case, we
say that $\ulam$ is an $a$-bipartition of $n$ and we also
write $\ulam\vdash n$.\smallskip

A partition $\lam$ is said to be $e$-restricted if
$0\leq\lam_i-\lam_{i+1}<e$ for all $i$. We say that a bipartition
$\ulam=(\lam^{(1)},\lam^{(2)})$ is $e$-restricted if both
$\lam^{(1)}$ and $\lam^{(2)}$ are $e$-restricted. Recall that for
each bipartition $\ulam=(\lam^{(1)},\lam^{(2)})$ of $n$, there is a
dual Specht module\footnote{Our $S_{\ulam}$ in this paper was
denoted by $\widetilde{S}^{\lam}$ in \cite{Hu1}.} $S_{\ulam}$ of
$\HH_q(B_n)$. If $\HH_q(B_n)$ is semisimple, then the set $
\bigl\{S_{\ulam}\bigm|\ulam=(\lam^{(1)},\lam^{(2)})\vdash n\bigr\} $
forms a complete set of pairwise non-isomorphic simple
$\HH_q(B_n)$-modules. In general, if $\ulam$ is $e$-restricted, then
$S_{\ulam}$ has a unique simple $\HH_q(B_n)$-head $D_{\ulam}$, and
the set
$$\Bigl\{D_{\ulam}\Bigm|\text{$\ulam=(\lam^{(1)},\lam^{(2)})\vdash
n$ is $e$-restricted}\Bigr\} $$ forms a complete set of pairwise
non-isomorphic simple $\HH_q(B_n)$-modules.\smallskip

Let $v$ be an indeterminate over $\Z$. Let $\mathcal{A}:=\Z[v,v^{-1}]$. For each bipartition
$\ulam=(\lam^{(1)},\lam^{(2)})$ of $n$, a Laurent polynomial
$f_{\ulam}(v)\in\mathcal{A}$ was introduced in \cite[Lemma
3.2]{Hu1}. If $\lam^{(1)}=\lam^{(2)}$, then it was proved in
\cite[Theorem 3.5]{Hu1} that there exist some Laurent polynomials
$g_{\ulam}(v)\in\mathcal{A}$, such that
$f_{\ulam}(v)=\bigl(g_{\ulam}(v)\bigr)^2$. Since $\cha K\neq 2$, we
actually have two choices for such $g_{\ulam}(v)$. {\it Henceforth,
for each bipartition $\ulam=(\lam^{(1)},\lam^{(2)})$ of $n$
satisfying $\lam^{(1)}=\lam^{(2)}$, we fix one choice of such
$g_{\ulam}(v)$, and we denote it by $\sqrt{f_{\ulam}(v)}$}. Then
another choice would be $-\sqrt{f_{\ulam}(v)}$. Once this is done, we can
canonically define two $\HH_q(D_n)$-submodules $S^{+}_{\ulam}$ and
$S^{-}_{\ulam}$ of $S_{\ulam}\!\downarrow_{\HH_q(D_n)}$ satisfying
$S_{\ulam}\!\downarrow_{\HH_q(D_n)}=S^{+}_{\ulam}\oplus
S^{-}_{\ulam}$ (see \cite[Theorem 4.6]{Hu3}). \smallskip

Let $\PP_n$ denote the set of bipartitions of $n$. For any
$\ulam, \umu\in\PP_n$, we define $\ulam\sim\umu$ if
$\lam^{(1)}=\mu^{(2)}$ and $\lam^{(2)}=\mu^{(1)}$. Let $\tau$ be the
$K$-algebra automorphism of $\HH_q(B_n)$ which is defined on
generators by $\tau(T_1)=T_0T_1T_0, \tau(T_i)=T_i$, for any $i\neq
1$. Let $\sigma$ be the $K$-algebra automorphism of $\HH_q(B_n)$
which is defined on generators by $\sigma(T_0)=-T_0,
\sigma(T_i)=T_i$, for any $i\neq 0$. Clearly
$\tau(\HH_q(D_n))=\HH_q(D_n)$ and
$\sigma\!\downarrow_{\HH_q(D_n)}=\id$. With our assumption
(\ref{separa}) in mind, by \cite[Corollary 3.7]{Hu3}, for each
$\ulam=(\lam^{(1)},\lam^{(2)})\in\PP_n$, $$
S_{\ulam}\!\downarrow_{\HH_q(D_n)}\cong
S_{\widehat{\ulam}}\!\downarrow_{\HH_q(D_n)},\,\,\text{where
$\widehat{\ulam}:=(\lam^{(2)},\lam^{(1)})$}.
$$

Note that the inequality (\ref{separa}) is only a part of the conditions for the semisimplicity of $\HH_q(D_n)$.
The following result depends heavily
on our assumption (\ref{separa}).

\addtocounter{theorem}{1}
\begin{lemma} {\rm (\cite{Hu3})} If $\HH_q(D_n)$ is semisimple, then the set $$
\Bigl\{S_{\ulam}\!\downarrow_{\HH_q(D_n)}, S^{+}_{(\beta,\beta)},
S^{-}_{(\beta,\beta)}\Bigm|\text{$\ulam=(\lam^{(1)},\lam^{(2)})\in\PP_n/\!{\sim}$,
$\lam^{(1)}\neq \lam^{(2)}$, $\beta\vdash m$}\Bigr\} $$ forms a
complete set of pairwise non-isomorphic simple $\HH_q(D_n)$-modules.
In general, if $\umu\vdash n$ is $e$-restricted and
$\mu^{(1)}\neq\mu^{(2)}$, then $D_{\umu}\!\downarrow_{\HH_q(D_n)}$
remains irreducible and it is the unique simple $\HH_q(D_n)$-head of
$S_{\umu}\!\downarrow_{\HH_q(D_n)}$; if $\alpha\vdash m$ is
$e$-restricted, then $S^{+}_{(\alpha,\alpha)}$ (resp.,
$S^{-}_{(\alpha,\alpha)}$) has a unique simple $\HH_q(D_n)$-head
$D^{+}_{(\alpha,\alpha)}$ (resp., $D^{-}_{(\alpha,\alpha)}$). The algebra
$\HH_q(D_n)$ is split over $K$ and the set
$$\biggl\{D_{\umu}\!\downarrow_{\HH_q(D_n)}, D^{+}_{(\alpha,\alpha)}, D^{-}_{(\alpha,\alpha)}\biggm|\begin{matrix}
\text{$\umu=(\mu^{(1)},\mu^{(2)})\in\PP_n/\!{\sim}$ is $e$-restricted}\\
\text{$\mu^{(1)}\neq\mu^{(2)}$, $\alpha\vdash m$ is
$e$-restricted}\end{matrix}\biggr\} $$ forms a complete set of
pairwise non-isomorphic simple $\HH_q(D_n)$-modules.
\end{lemma}

Therefore, we can regard the following set $$
\biggl\{S_{\ulam}\!\downarrow_{\HH_q(D_n)}, S^{+}_{(\beta,\beta)},
S^{-}_{(\beta,\beta)}\biggm|\begin{matrix}
\text{$\ulam=(\lam^{(1)},\lam^{(2)})\in\PP_n/\!{\sim}, \lam^{(1)}\neq\lam^{(2)}$, and }\\
\text{$\beta$ is a partition of $m$}\end{matrix}\biggr\}
$$
as the set of dual Specht modules for $\HH_q(D_n)$ in the separated case.\smallskip

We collect together some facts in the following lemma.

\begin{lemma} {\rm (\cite{Hu1}, \cite{Hu3})} Let $\beta$ be a partition of $m$ and $\alpha$ be an $e$-restricted
partition of $m$. Let $\ulam$ be a bipartition of $n$ and $\umu$ be an
$e$-restricted bipartition of $n$. Then we have
\begin{enumerate}
\item[(2.3.1)] $f_{\ulam}(q)$ is an invertible element in $K$;
\item[(2.3.2)] $\bigl(D_{\umu}\bigr)^{\sigma}\cong D_{\widehat{\umu}}$, $D_{\umu}\!\downarrow_{\HH_q(D_n)}\cong
D_{\widehat{\umu}}\!\downarrow_{\HH_q(D_n)}$;
\item[(2.3.3)] $D_{(\alpha,\alpha)}\!\downarrow_{\HH_q(D_n)}\cong D^{+}_{(\alpha,\alpha)}\oplus D^{-}_{(\alpha,\alpha)}$;
\item[(2.3.4)] $\bigl(S^{+}_{(\beta,\beta)}\bigr)^{\tau}\cong S^{-}_{(\beta,\beta)}$,
$\bigl(D^{+}_{(\alpha,\alpha)}\bigr)^{\tau}\cong
D^{-}_{(\alpha,\alpha)}$;
\item[(2.3.5)] $M^{\tau}\cong M$ for any $\HH_q(B_n)$-module $M$.
\end{enumerate}
\end{lemma}
Note that except for (2.3.5), all the claims in the above lemma depend on the validity of our assumption (\ref{separa}).
\smallskip

Recall the Morita equivalence (\ref{MoritaEquivalence}) proved in
\cite{Hu1}. Let $F$ be the resulting functor from the category of
finite dimensional $\HH_q(D_n)$-modules to the category of finite
dimensional modules over
$A(m)\oplus\bigoplus_{a=m+1}^{n}\HH_q(\BS_{(a,n-a)})$. For any
$\ulam, \umu\in\PP_n$ with $\umu$ being $e$-restricted, we define $$
S(\ulam):=S_{\lam^{(1)}}\otimes \widehat{S}_{\lam^{(2)}},\quad
D(\umu):=D_{\mu^{(1)}}\otimes\widehat{D}_{\mu^{(2)}},
$$
if $|\lam^{(1)}|\neq |\lam^{(2)}|, |\mu^{(1)}|\neq |\mu^{(2)}|$; and
$$
S(\ulam):=\bigl(S_{\lam^{(1)}}\otimes
\widehat{S}_{\lam^{(2)}}\bigr)\uparrow_{\HH_q(\BS_{(m,m)})}^{A(m)},\quad
D(\umu):=\bigl(D_{\mu^{(1)}}\otimes
\widehat{D}_{\mu^{(2)}}\bigr)\uparrow_{\HH_q(\BS_{(m,m)})}^{A(m)},
$$ if $|\lam^{(1)}|=|\lam^{(2)}|=m=|\mu^{(1)}|=|\mu^{(2)}|$.

For each integer $i$ with $1\leq i\leq m-1$, we define
$s_i=(i,i+1)$. Then $\{s_1,s_2,\cdots,s_{m-1}\}$ is the set of all the
simple reflections in $\BS_m$. A word $w=s_{i_1}\cdots s_{i_k}$ for
$w\in\BS_{m}$ is a reduced expression if $k$ is minimal; in this
case we say that $w$ has length $k$ and we write $\ell(w)=k$. Given
a reduced expression $s_{i_1}\cdots s_{i_k}$ for $w\in\BS_m$, we
write $T_w=T_{i_1}\cdots T_{i_k}$. It is well-known that
$\{T_w|w\in\BS_m\}$ forms a $K$-basis of $\HH_q(\BS_m)$.
Let $\beta$ be a partition of $m$. Let $\BS_{\beta}$ be the Young subgroup of $\BS_m$ corresponding to
$\beta$. We set $$
x_{\beta}:=\sum_{w\in\BS_{\beta}}T_w,\quad y_{\beta}:=\sum_{w\in\BS_{\beta}}(-q)^{-\ell(w)}T_w.
$$
Let $\ft^{\beta}$ (resp., $\ft_{\beta}$) be the standard $\beta$-tableau in which the numbers $1,2,\cdots,m$ appear in
order
along successive rows (resp., columns). Let $w_{\beta}\in\BS_{m}$ be such that $\ft^{\beta}w_{\beta}=\ft_{\beta}$.
If $\alpha$ is an $e$-restricted partition of $m$, then we have the following direct sum decompositions of
$A(m)$-modules: $$
S(\beta,\beta)=S(\beta,\beta)_{+}\oplus S(\beta,\beta)_{-},\quad D(\alpha,\alpha)=D(\alpha,\alpha)_{+}
\oplus D(\alpha,\alpha)_{-},
$$
where $$\begin{aligned}
S(\beta,\beta)_{\pm}:&=\Bigl(\sqrt{f_{(\beta,\beta)}(q)}z'_{\beta}\otimes \widehat{z'_{\beta}}\pm (z'_{\beta}\otimes
\widehat{z'_{\beta}})h(m)\Bigr)\HH_q(\BS_{(m,m)}),\\
D(\alpha,\alpha)_{\pm}:&=\Bigl(\sqrt{f_{(\alpha,\alpha)}(q)}\overline{z'_{\alpha}}\otimes
\widehat{\overline{z'_{\alpha}}}\pm (\overline{z'_{\alpha}}\otimes \widehat{\overline{z'_{\alpha}}})h(m)\Bigr)
\HH_q(\BS_{(m,m)}),
\end{aligned}
$$
and $z'_{\beta}:=y_{\beta'}T_{w_{\beta'}}x_{\beta}$, $\beta'$ is the conjugate partition of $\beta$,
$\widehat{z'_{\beta}}$ is the similarly defined generator for the
dual Specht module $\widehat{S}_{\beta}$ of
$\HH_q(\BS_{(m+1,\cdots,n)})$, $\overline{z'_{\alpha}}$ and
$\widehat{\overline{z'_{\alpha}}}$ are the canonical images of
$z'_{\alpha}$ and $\widehat{z'_{\alpha}}$ in $D_{\alpha}$ and $\widehat{D}_{\alpha}$ respectively. Note that by
\cite[Theorem 3.5]{DJ2} and \cite[(5.2), (5.3)]{Mu}, the right ideal of $\HH_q(\BS_m)$ generated  by $z'_{\beta}$ is
isomorphic to the dual Specht module ${S}_{\beta}$, and the dual Specht module ${S}_{\beta}$ is also isomorphic to the
right ideal of $\HH_q(\BS_m)$ generated by
$x_{\beta'}^{\#}T_{w_{\beta'}}^{\#}y_{\beta}^{\#}$. Here ``$\#$" denotes the $K$-algebra automorphism of $\HH_q(\BS_n)$
which is defined on generators by $T_i^{\#}
:=(-q)T_i^{-1}$ for each integer $i$ with $1\leq i\leq m-1$. We take this chance to point out that the
definition of $S(\lam)_{\pm}$ and $D(\mu)_{\pm}$ given in \cite[p. 428, lines 25, 26]{Hu1} are not correct (although this does not
affect any other results in \cite{Hu1}). The correct definition
should be as what we have given above. By direct verification, we
see that $$ F\bigl(S_{(\beta,\beta)}^{\pm}\bigr)\cong
S(\beta,\beta)_{\pm},\quad
F\bigl(D_{(\alpha,\alpha)}^{\pm}\bigr)\cong D(\alpha,\alpha)_{\pm},
$$
and for any $\ulam, \umu\in\PP_n$ with $\umu$ being $e$-restricted,
$$ F\bigl(S_{\ulam}\!\downarrow_{\HH_q(D_n)}\bigr)\cong
S(\ulam),\quad F\bigl(D_{\umu}\!\downarrow_{\HH_q(D_n)}\bigr)\cong
D(\umu).
$$
The following lemma is a direct consequence of this Morita equivalence.

\begin{lemma} \label{decnumodd1} Let $\ulam$ be a bipartition of $n$ and $\umu$ an $e$-restricted bipartition of $n$. Let
$\beta$ be a partition of $m$ and $\alpha$ an $e$-restricted partition of $m$. Then we have \begin{enumerate}
\item[(2.4.1)] if $|\lam^{(1)}|\neq |\lam^{(2)}|$ and $\mu^{(1)}\neq\mu^{(2)}$, then $$\begin{aligned} &
\bigl[S_{\ulam}\downarrow_{\HH_q(D_n)}:D_{\umu}\downarrow_{\HH_q(D_n)}\bigr]=\begin{cases}\bigl[S_{\lam^{(1)}}:
D_{\mu^{(1)}}\bigr]
    \bigl[\widehat{S}_{\lam^{(2)}}:
    \widehat{D}_{\mu^{(2)}}\bigr], &\text{if \text{$\begin{matrix}\text{$|\lam^{(1)}|=|\mu^{(1)}|$ $\&$}\\
    \text{$|\lam^{(2)}|=|\mu^{(2)}|$}\end{matrix}$,}}\\
    0, &\text{otherwise;}
    \end{cases}\\
    & \bigl[S_{\ulam}\downarrow_{\HH_q(D_n)}:D^{+}_{(\alpha,\alpha)}\bigr]=0=\bigl[S_{\ulam}\downarrow_{\HH_q(D_n)}:
    D^{-}_{(\alpha,\alpha)}\bigr];
\end{aligned} $$
\item[(2.4.2)] if $|\lam^{(1)}|=|\lam^{(2)}|=m$, $\lam^{(1)}\neq\lam^{(2)}$ and $|\mu^{(1)}|\neq |\mu^{(2)}|$, then
$$\begin{aligned}
&\bigl[S_{\ulam}\downarrow_{\HH_q(D_n)}:D_{\umu}\downarrow_{\HH_q(D_n)}\bigr]=0;\\
&\bigl[S^{+}_{(\beta,\beta)}:D_{\umu}\downarrow_{\HH_q(D_n)}\bigr]=0=\bigl[S^{-}_{(\beta,\beta)}:
D_{\umu}\downarrow_{\HH_q(D_n)}\bigr].
\end{aligned}$$
\end{enumerate}
\end{lemma}

\begin{theorem} \label{decnumodd2} Let $\ulam$ be a bipartition of $n$ with $|\lam^{(1)}|=|\lam^{(2)}|=m$. Let $\umu$
be an $e$-restricted bipartition of $n$
with $|\mu^{(1)}|=|\mu^{(2)}|=m$. Let $\beta$ be a partition of $m$
and $\alpha$ an $e$-restricted partition of $m$. Then we have
\begin{enumerate}
\item[(2.5.1)] if $\lam^{(1)}\neq\lam^{(2)}$ and $\mu^{(1)}\neq\mu^{(2)}$, then $$\begin{aligned} &
\bigl[S_{\ulam}\downarrow_{\HH_q(D_n)}:D_{\umu}\downarrow_{\HH_q(D_n)}\bigr]=
    \bigl[S_{\lam^{(1)}}:D_{\mu^{(1)}}\bigr]\bigl[\widehat{S}_{\lam^{(2)}}:
    \widehat{D}_{\mu^{(2)}}\bigr]+\bigl[S_{\lam^{(1)}}:D_{\mu^{(2)}}\bigr]\bigl[\widehat{S}_{\lam^{(2)}}:
    \widehat{D}_{\mu^{(1)}}\bigr], \\
    & \bigl[S_{\ulam}\downarrow_{\HH_q(D_n)}:D^{+}_{(\alpha,\alpha)}\bigr]=\bigl[S_{\ulam}\downarrow_{\HH_q(D_n)}:
    D^{-}_{(\alpha,\alpha)}\bigr]=
    \bigl[S_{\lam^{(1)}}:D_{\alpha}\bigr]\bigl[\widehat{S}_{\lam^{(2)}}:
    \widehat{D}_{\alpha}\bigr]/2;\end{aligned} $$
\item[(2.5.2)] if $\mu^{(1)}\neq\mu^{(2)}$, then
$$\begin{aligned} &   \bigl[S^{+}_{(\beta,\beta)}:D_{\umu}\downarrow_{\HH_q(D_n)}\bigr]=
    \bigl[S_{\beta}:D_{\mu^{(1)}}\bigr]\bigl[\widehat{S}_{\beta}:
    \widehat{D}_{\mu^{(2)}}\bigr]/2+\bigl[S_{\beta}:D_{\mu^{(2)}}\bigr]\bigl[\widehat{S}_{\beta}:
    \widehat{D}_{\mu^{(1)}}\bigr]/2,\\
    & \bigl[S^{-}_{(\beta,\beta)}:D_{\umu}\downarrow_{\HH_q(D_n)}\bigr]=\bigl[S^{+}_{(\beta,\beta)}:
    D_{\umu}\downarrow_{\HH_q(D_n)}\bigr];
    \end{aligned}
$$
\end{enumerate}
\end{theorem}

\begin{proof} Using the functor $F$, it is easy to see that the first equality in (2.5.1) follows from the exactness of
the induction functor
$\uparrow\!_{\HH_q(\BS_{(m,m)})}^{A(m)}$. The other equalities follow from (2.3.4), (2.3.5) and the following Morita
equivalence (\cite{DJ3}): \addtocounter{equation}{4}
\begin{equation}\label{MoritaB}
\HH_{q}(B_n)\overset{Morita}{\sim}\bigoplus_{a=0}^{n}\HH_q(\BS_{(a,n-a)}).
\end{equation}
\end{proof}

It remains to determine the decomposition numbers $$
\bigl[S^{+}_{(\beta,\beta)}:D^{\pm}_{(\alpha,\alpha)}\bigr],\quad
\bigl[S^{-}_{(\beta,\beta)}:D^{\pm}_{(\alpha,\alpha)}\bigr].
$$

\addtocounter{theorem}{1}
\begin{lemma} \label{decnumeven1} Let $\beta$ be a partition of $m$ and $\alpha$ an $e$-restricted partition of $m$.
Then we have \begin{enumerate}
\item[(2.7.1)] $\bigl[S^{+}_{(\beta,\beta)}:D^{+}_{(\alpha,\alpha)}\bigr]=\bigl[S^{-}_{(\beta,\beta)}:D^{-}_{(\alpha,
\alpha)}\bigr]$;
\item[(2.7.2)] $\bigl[S^{+}_{(\beta,\beta)}:D^{-}_{(\alpha,\alpha)}\bigr]=\bigl[S^{-}_{(\beta,\beta)}:D^{+}_{(\alpha,
\alpha)}\bigr]$;
\item[(2.7.3)] $\bigl[S^{+}_{(\beta,\beta)}:D^{+}_{(\alpha,\alpha)}\bigr]+\bigl[S^{+}_{(\beta,\beta)}:D^{-}_{(\alpha,
\alpha)}\bigr]=
\bigl[S_{\beta}:D_{\alpha}\bigr]^2$.
\end{enumerate}
\end{lemma}
\begin{proof} The first two equalities follow from (2.3.4), while the last equality follows from (\ref{MoritaB}) and
(2.3.4) and the fact that
$$S_{(\beta,\beta)}\!\downarrow_{\HH_q(D_n)}\cong S^{+}_{(\beta,\beta)}\oplus S^{-}_{(\beta,\beta)},\quad
D_{(\alpha,\alpha)}\!\downarrow_{\HH_q(D_n)}\cong
D^{+}_{(\alpha,\alpha)}\oplus D^{-}_{(\alpha,\alpha)}.$$
\end{proof}

Therefore, it suffices to determine the decomposition number
$\bigl[S^{+}_{(\beta,\beta)}:D^{+}_{(\alpha,\alpha)}\bigr]$ for each
partition $\beta$ of $m$ and each $e$-restricted partition $\alpha$
of $m$. Note that these decomposition numbers are the $2$-splittable
decomposition numbers in the sense of \cite{HMa}.
\smallskip

The purpose of this article is to give an explicit formula for these
decomposition numbers. We shall relate these decomposition numbers
with the decomposition numbers of the Hecke algebra $\HH_q(\BS_{m})$
and certain Schur elements of $\HH_q(B_n)$ and of $\HH_q(\BS_{m})$. The
main results of this paper are the following two theorems:

\begin{theorem} \label{decnum0} Let $\beta$ be a partition of $m$ and $\alpha$ an $e$-restricted partition of $m$. Let $$
d_{\beta,\alpha}:=\bigl[S_{\beta}:D_{\alpha}\bigr].
$$
Then we have the following equality in $K$: $$
\bigl[S^{+}_{(\beta,\beta)}:D^{+}_{(\alpha,\alpha)}\bigr]=d_{\beta,\alpha}
\biggl(\frac{\sqrt{f_{(\beta,\beta)}(q)}}{\sqrt{f_{(\alpha,\alpha)}(q)}}
+d_{\beta,\alpha}\biggr)/2.
$$
In particular, if $\cha K=0$, then the above equality completely determines the
decomposition number.
\end{theorem}

Note that {\it a priori}\, we do not even know why the righthand
side term in the above equality should be an element in the prime
subfield of $K$. Note also that we allow $q=1$ in the above theorem.
The Laurent polynomials $f_{(\beta,\beta)}(v),
f_{(\alpha,\alpha)}(v)$ appeared in the above theorem can be
computed explicitly by the next theorem.

\begin{theorem} \label{mainthm} Let $\ulam=(\lam^{(1)},\lam^{(2)})$ be an arbitrary bipartition of $n$. Then we have $$
f_{\ulam}(v)=v^{\frac{n(n-1)}{2}}\frac{s_{\ulam}(v,1)}{s_{\lam^{(1)}}(v)s_{\lam^{(2)}}(v)},
$$
where $s_{\ulam}(v,\tilde{v})$ (resp., $s_{\lam^{(1)}}(v)$, $s_{\lam^{(2)}}(v)$) is the Schur element corresponding to
$\ulam$ (resp., corresponding to $\lam^{(1)}$, $\lam^{(2)}$), and $\tilde{v}$ is another indeterminate over $\Z$.
\end{theorem}

Note that these Schur elements are some explicit defined Laurent polynomials on $v, \tilde{v}$. For example, if $\lam$
is a partition, then $$
s_{\lam}(v)=v^{-\ell(w_{\lam',0})}\prod_{(i,j)\in[\lam]}[h_{i,j}^{\lam}]_v,
$$
where $h_{i,j}^{\lam}=\lam_i+\lam'_j-i-j+1$ (the $(i,j)$th hook
length), $w_{\lam',0}$ is the unique longest element in $\BS_{\lam'}$, $\lam'$ denotes the conjugate partition of
$\lam$, and for each integer $k$,
$$[k]_v:=\frac{v^k-1}{v-1}\in\mathcal{A}. $$
We refer the reader to \cite{Ma1} for the explicit definitions of Schur elements corresponding to arbitrary
multi-partitions.

\begin{example} Suppose that $K=\mathbb{C}, n=6$, $\beta=(2,1), \alpha=(1,1,1)$. If $e=3$, then $\alpha$ is $e$-restricted and
the assumption (\ref{separa}) is satisfied. In this case, let $q$ be a primitive $3$th root of unity in
$\mathbb{C}$. Then $1+q+q^2=0$. It is well-known that $[S_{\beta}:D_{\alpha}]=1$. Applying Theorem \ref{mainthm} and the known formulae for Schur elements, we get that $$\begin{aligned}
f_{(\beta,\beta)}(v)&=v^{4}(v+1)^{4}(v^3+1)^2,\\
f_{(\alpha,\alpha)}(v)&=(v+1)^2(v^2+1)^2(v^3+1)^2.
\end{aligned}
$$
Applying Theorem \ref{decnum0}, we get that $$
\bigl[S^{+}_{(\beta,\beta)}:D^{+}_{(\alpha,\alpha)}\bigr]=\frac{1}{2}\biggl(\frac{q^2(q+1)^2(q^3+1)}{(q+1)(q^2+1)(q^3+1)}+
1\biggr)=1;
$$
If $e=5$, then $\alpha$ is also $e$-restricted and the assumption
(\ref{separa}) is still satisfied. In that case,
$[S_{\beta}:D_{\alpha}]=0$, and hence $$
\bigl[S^{+}_{(\beta,\beta)}:D^{+}_{(\alpha,\alpha)}\bigr]=0.
$$
\end{example}

\bigskip\bigskip
\section{Lifting to $q$-Schur algebras}

For each integer $i\in\{1,2,\cdots,n-1\}\setminus\{m\}$, we define
$$
\widehat{i}=\begin{cases} {i+m} &\text{{\rm if}\,\,\, $1\leq i<m$,}\\
{i-m} &\text{{\rm if}\,\,\, $m<i\leq n-1$.}\end{cases} $$
Let \,\,\,$\widehat{\null}$\,\,\,\, be the group automorphism of
$\BS_{(m,m)}:=\BS_m\times\BS_{\{m+1,\cdots,n\}}$ which is defined on
generators by $\widehat{s_i}=s_{\widehat{i}}$ for each integer $i\in\{1,2,\cdots,n-1\}\setminus\{m\}$. We use
the same notation \,\,\,$\widehat{\null}$\,\,\, to denote the algebra
automorphism of the Hecke algebra
$\HH_q(\BS_{(m,m)})=\HH_q(\BS_m)\otimes\HH_q(\BS_{\{m+1,\cdots,n\}})$
which is defined on generators by $\widehat{T_i}=T_{\widehat{i}}$
for each integer $i\in\{1,2,\cdots,n-1\}\setminus\{m\}$.\smallskip

For each partition $\lam$ of $n$, we set $$
x_{\lam}:=\sum_{w\in\BS_{\lam}}T_w,
$$
where $\BS_{\lam}$ is the Young subgroup of $\BS_n$ corresponding to
$\lam$. We use $\Lambda(n)$ to denote the set of compositions
$\lam=(\lam_1,\lam_2,\cdots,\lam_n)$ of $n$ into $n$ parts (each
part $\lam_i$ being nonnegative). Let $\Lambda^{+}(n)$ be the set of partitions in $\Lambda(n)$. Following Dipper and
James
\cite{DJ4}, \cite{DJ5}, we define the $q$-Schur algebra $S_q(n)$ to be $$
S_q(n):=\End_{\HH_q(\BS_n)}\Bigl(\bigoplus_{\lam\in\Lambda(n)}x_{\lam}\HH_q(\BS_n)\Bigr).
$$
In a similar way, we define the $q$-Schur algebras $S_q(m),
\widehat{S}_q(m)$ by using the Hecke algebras $\HH_q(\BS_m),
\HH_q(\BS_{\{m+1,\cdots,n\}})$  respectively. Note that to define the
$q$-Schur algebra $\widehat{S}_q(m)$, one needs to use the element
$\widehat{x}_{\lam}$ for $\lam\in\Lambda(m)$.\smallskip

For any $\lam,\mu\in\Lambda(n)$, let $\DD_{\lam,\mu}$ be the set
of distinguished $\BS_{\lam}$-$\BS_{\mu}$-double coset
representatives in $\BS_n$. Following \cite{DJ5}, for each
$d\in\DD_{\lam,\mu}$, we define
$$\phi_{\lam,\mu}^{d}\in\Hom_{\HH_q(\BS_n)}\bigl(x_{\mu}\HH_q(\BS_n),x_{\lam}\HH_q(\BS_n)\bigr)$$
by: $$ \phi_{\lam,\mu}^{d}\bigl(x_{\mu}h\bigr)=\sum_{w\in
\BS_{\lam}d\BS_{\mu}}T_w h,\quad\forall\,h\in\HH_q(\BS_n).
$$
By \cite[Theorem 1.4]{DJ5}, the elements in the set $$
\bigl\{\phi_{\lam,\mu}^{d}\bigm|\lam,\mu\in\Lambda(n),
d\in\DD_{\lam,\mu}\bigr\}
$$
form a $K$-basis of $S_q(n)$ which shall be called {\it standard
bases} in this paper.\smallskip

For any $\lam,\mu\in\Lambda(n)$, by \cite[(2.3)]{DJ5}, $$
\phi_{\lam,\lam}^{1}S_q(n)\phi_{\mu,\mu}^{1}\cong\Hom_{\HH_q(\BS_n)}\bigl(x_{\mu}\HH_q(\BS_n),x_{\lam}\HH_q(\BS_n)\bigr).
$$
Let $\omega_n$ denote the partition
$(1^n)=(\underbrace{1,1,\cdots,1}_{\text{$n$ copies}})$ of $n$.
Using the natural isomorphism $$
\Hom_{\HH_q(\BS_n)}\bigl(\HH_q(\BS_n),\HH_q(\BS_n)\bigr)\cong
\HH_q(\BS_n),
$$
we can identify $\HH_q(\BS_n)$ with the non-unital $K$-subalgebra
$\phi_{\omega_n,\omega_n}^{1}S_q(n)\phi_{\omega_n,\omega_n}^{1}$ of
$S_q(n)$. We use $\iota_n$ to denote the resulting injection
from $\HH_q(\BS_n)$ into $S_q(n)$. In a similar way, we can identify
$\HH_q(\BS_m)$ with the non-unital $K$-subalgebra
$\phi_{\omega_m,\omega_m}^{1}S_q(m)\phi_{\omega_m,\omega_m}^{1}$
via an injective map $\iota_m$, and $\HH_q(\BS_{\{m+1,\cdots,n\}})$
with the non-unital $K$-subalgebra $\widehat{\phi}_{\omega_m,\omega_m}^{1}
\widehat{S}_q(m)\widehat{\phi}_{\omega_m,\omega_m}^{1}$ via an injective map
$\widehat{\iota}_m$. Note that here in order
not to confuse with the standard basis elements of $S_q(m)$, we denote the
standard basis element of $\widehat{S}_q(m)$ by
$\widehat{\phi}_{\lam,\mu}^{d}$.
\smallskip

Let $\rho$ denote the natural injective map from
$\HH_q(\BS_m)\otimes\HH_q(\BS_{\{m+1,\cdots,n\}})$ into
$\HH_q(\BS_n)$. We are going to lift this map to an injection
$\widetilde{\rho}$ from $S_q(m)\otimes\widehat{S}_q(m)$ into
$S_q(n)$. Let $\DD_{(m,m)}$ be the set of distinguished right coset
representatives of $\BS_{(m,m)}$ in $\BS_n$. Any element
$h\in\HH_q(\BS_n)$ can be written uniquely as
\begin{equation}\label{anyH} h=\sum_{\substack{w_1\in\BS_m,
w_2\in\BS_{\{m+1,\cdots,n\}}\\
d\in\DD_{(m,m)}}}A_{w_1,w_2}^{d}T_{w_1}T_{w_2}T_d,
\end{equation}
where $A_{w_1,w_2}^{d}\in K$ for each $w_1, w_2, d$. For any $f\in S_q(m), g\in\widehat{S}_q(m)$, we define
$\widetilde{\rho}(f\otimes g)\in S_q(n)$ as follows: for any $\mu\in\Lambda(n)$, and any $h\in\HH_q(\BS_n)$ which is
given by (\ref{anyH}), if $\mu=(\mu^{(1)},\mu^{(2)})$, where $\mu^{(1)}\in\Lambda(m)$, then we set $$
\widetilde{\rho}(f\otimes g)\bigl(x_{\mu}h\bigr):=\sum_{\substack{w_1\in\BS_m, w_2\in\BS_{\{m+1,\cdots,n\}}\\
d\in\DD_{(m,m)}}}A_{w_1,w_2}^{d}f\bigl(x_{\mu^{(1)}}T_{w_1}\bigr)g\bigl(x_{\widehat{\mu}^{(2)}}T_{w_2}\bigr)T_d;
$$
otherwise, we set $\widetilde{\rho}(f\otimes g)\bigl(x_{\mu}h\bigr):=0$.

\addtocounter{theorem}{1}

\begin{lemma} With the notations as above, we have that $$
\widetilde{\rho}(f\otimes g)\in S_q(n).
$$
\end{lemma}
\begin{proof} This is clear by using the fact that $$
\bigl(x_{\mu^{(1)}}\HH_q(\BS_m)\otimes \widehat{x}_{\mu^{(2)}}
\HH_q(\BS_{\{m+1,\cdots,n\}})\bigr)\!\uparrow\!_{\HH_q(\BS_{(m,m)})}^{\HH_q(\BS_n)}\cong x_{\mu}\HH_q(\BS_n).
$$
\end{proof}

\begin{lemma} We have the following commutative diagram of algebra maps: $$
\begin{CD}
\HH_q(\BS_m)\otimes\HH_q(\BS_{\{m+1,\cdots,n\}})
@>{\rho}>>\HH_q(\BS_n)\\
@V{\iota_m\otimes\widehat{\iota}_m} VV @V{\iota_n}VV\\
S_q(m)\otimes \widehat{S}_q(m)@>{\widetilde{\rho}}>>S_q(n)
\end{CD}
$$
Moreover, $\widetilde{\rho}(\phi_{\omega_m,\omega_m}^{1}\otimes
\widehat{\phi}_{\omega_m,\omega_m}^{1})=\phi_{\omega_n,\omega_n}^{1}$
and $\widetilde{\rho}$ is an injection.\end{lemma}
\begin{proof} This follows from direct verification.
\end{proof}

Note that the unit element of $S_q(m)\otimes \widehat{S}_q(m)$, i.e., $$
\bigl(\sum_{\lam\in\Lambda(m)}\phi_{\lam,\lam}^{1}\bigr)\otimes\bigl(\sum_{\lam\in\Lambda(m)}
\widehat{\phi}_{\lam,\lam}^{1}\bigr),
$$
is not mapped by $\widetilde{\rho}$ to the unit element of $S_q(n)$. Henceforth, we identify $S_q(m)\otimes
\widehat{S}_q(m)$ with a non-unital
$K$-subalgebra of $S_q(n)$ via the injection $\widetilde{\rho}$. Recall
that (\cite[Remark 2.4]{Hu1}) the algebra $A(m)$ was generated by
$T_{1},\cdots,T_{m-1}$, $T_{m+1},\cdots,T_{n-1},h(m)$
and satisfy the following relations
$$\begin{aligned}
&(T_i+1)(T_i-q)=0,\quad\text{for $1\leq i\leq n-1, i\neq m$,}\qquad h(m)^2=z_{m,m},\\
&T_iT_{i+1}T_i=T_{i+1}T_{i}T_{i+1},\quad\text{for $1\leq i\leq n-2$, $i\notin\{m-1,m\}$}\\
&T_iT_j=T_jT_i,\quad\text{for $1\leq i<j-1\leq n-2$, $i,j\neq m$}\\
&T_{i}h(m)=\begin{cases} h(m)T_{i+m} &\text{{\rm if}\,\,\, $1\leq i<m$,}\\
h(m)T_{i-m} &\text{{\rm if}\,\,\, $m<i\leq n-1$,}\end{cases}
\end{aligned}
$$
where $z_{m,m}$ is a central element in the Hecke algebra
$\HH_q(\BS_{(m,m)})$ (see \cite{DJ3}). We are going to lift
the elements $h(m), z_{m,m}$ to the corresponding $q$-Schur
algebras.

\begin{lemma} Let $z$ be an element in the center of $\HH_q(\BS_n)$. We define $Z\in S_q(n)$ by $$
Z\bigl(x_{\lam}h\bigr)=x_{\lam}hz,\quad \forall\,\lam\in\Lambda(n), h\in\HH_q(\BS_n).
$$
Then $Z$ lies in the center of $S_q(n)$.
\end{lemma}
\begin{proof} The condition that $z$ is central in $\HH_q(\BS_n)$ implies that the above-defined map $Z$ is indeed a
right $\HH_q(\BS_n)$-homomorphism. Hence $Z\in S_q(n)$. It remains to check that $\phi_{\lam,\mu}^{d}Z=Z\phi_{\lam,\mu}^{d}$ for any $\lam,\mu\in\Lambda(n)$ and
any $d\in\DD_{\lam,\mu}$.
Applying \cite[(2.1)]{DJ5}, and by definition, for any $h\in\HH_q(\BS_n)$, $$\begin{aligned}
\phi_{\lam,\mu}^{d}Z\bigl(x_{\mu}h\bigr)
&=\bigl(\phi_{\lam,\mu}^{d}\bigr)(x_{\mu}hz)=\sum_{w\in
\BS_{\lam}d\BS_{\mu}}T_w hz,\\
Z\phi_{\lam,\mu}^{d}\bigl(x_{\mu}h\bigr)
&=Z\sum_{w\in
\BS_{\lam}d\BS_{\mu}}T_w h=\sum_{w\in
\BS_{\lam}d\BS_{\mu}}T_w hz.
\end{aligned}
$$
This proves that $\phi_{\lam,\mu}^{d}Z=Z\phi_{\lam,\mu}^{d}$, as required.
\end{proof}

\begin{definition} We define $H(m), Z_{m,m}\in S_q(n)$ as follows: for any $\lam\in\Lambda(n), h\in\HH_q(\BS_n)$,
$$\begin{aligned}
H(m)(x_{\lam}h):&=\begin{cases}x_{\widehat{\lam}}h(m)h, &\text{if $\lam=(\lam^{(1)},\lam^{(2)})$ with $\lam^{(1)}\in
\Lambda(m)$,}\\
0, &\text{otherwise,}\end{cases}\\
Z_{m,m}(x_{\lam}h):&=\begin{cases}x_{{\lam}}z_{m,m}h, &\text{if $\lam=(\lam^{(1)},\lam^{(2)})$ with $\lam^{(1)}\in
\Lambda(m)$,}\\
0, &\text{otherwise.}\end{cases}
\end{aligned}
$$
\end{definition}
\noindent

\begin{lemma} \label{lm36} With the notations as above, we have that $$
H(m)^2=Z_{m,m},
$$
and $Z_{m,m}$ lies in the center of $S_q(m)\otimes\widehat{S}_q(m)$. Moreover $Z_{m,m}$ is invertible in
$S_q(m)\otimes\widehat{S}_q(m)$.
\end{lemma}
\begin{proof} The equality $H(m)^2=Z_{m,m}$ follows from direct verification and the fact that $h(m)^2=z_{m,m}$.

Since $z_{m,m}$ lies in $\HH_q(\BS_{(m,m)})$ and
$z_{m,m}$ is invertible in $\HH_q(\BS_{(m,m)})$ (see \cite[(4.12)]{DJ3}), it follows that
$Z_{m,m}\in S_q(m)\otimes\widehat{S}_q(m)$ and $Z_{m,m}$ is
invertible in $S_q(m)\otimes\widehat{S}_q(m)$. The inverse of
$Z_{m,m}$ is given by $$
Z_{m,m}^{-1}(x_{\lam}h):=\begin{cases}x_{{\lam}}z_{m,m}^{-1}h,
&\text{if $\lam=(\lam^{(1)},
\lam^{(2)})$ with $\lam^{(1)}\in\Lambda(m)$,}\\
0, &\text{otherwise,}\end{cases}
$$
for any $\lam\in\Lambda(n), h\in\HH_q(\BS_n)$. The claim that
$Z_{m,m}$ lies in the center of $S_q(m)\otimes\widehat{S}_q(m)$ also
follows from direct verification.
\end{proof}

Note that $Z_{m,m}$ is invertible in $S_q(m)\otimes\widehat{S}_q(m)$
does not mean that $Z_{m,m}$ is invertible in $S_q(n)$. The point is
that the unit element of $S_q(n)$ is not the same as the unit
element of $S_q(m)\otimes\widehat{S}_q(m)$. By the same reason, we
know that $H(m)$ is not an invertible element in $S_q(n)$.

We identify $\HH_q(\BS_{(m,m)})$ with a subalgebra of
$S_q(m)\otimes\widehat{S}_q(m)$ via the injection
$\iota_m\otimes\widehat{\iota}_m$. For any
$\lam,\mu\in\Lambda(m)$, let $\mathcal{D}_{\lam,\mu}$ be the set
of distinguished $\BS_{\lam}$-$\BS_{\mu}$-double coset
representatives in $\BS_m$. For any $d\in\mathcal{D}_{\lam,\mu}$, it
is easy to see that $\widehat{d}$ is a distinguished
$\widehat{\BS}_{\lam}$-$\widehat{\BS}_{\mu}$-double coset
representative in $\BS_{\{m+1,\cdots,n\}}$. Here
$\widehat{\BS}_{\lam}, \widehat{\BS}_{\mu}$ denote the image of
$\BS_{\lam}, \BS_{\mu}$ under the automorphism
\,\,\,$\widehat{\null}$\,\,\,.

\begin{lemma} \label{lm37} The automorphism \,\,$\widehat{\null}$\,\, of
$\HH_q(\BS_{(m,m)})$ can be uniquely extended to a $K$-algebra
automorphism (still denoted by \,\,$\widehat{\null}$\,\,) of
$S_q(m)\otimes\widehat{S}_q(m)$ such that for any
$\lam,\mu\in\Lambda(m)$ and any $d\in\mathcal{D}_{\lam,\mu}$, $$
\widehat{\phi_{\lam,\mu}^d}=\widehat{\phi}_{\lam,\mu}^{\widehat{d}}.
$$
Furthermore,
$H(m)\phi_{\lam,\mu}^d=\widehat{\phi}_{\lam,\mu}^{\widehat{d}}H(m)$
and
$\phi_{\lam,\mu}^dH(m)=H(m)\widehat{\phi}_{\lam,\mu}^{\widehat{d}}$.
\end{lemma}
\begin{proof} For the first claim, it suffices to show that the map
which sends $\phi_{\lam,\mu}^d$ to
$\widehat{\phi}_{\lam,\mu}^{\widehat{d}}$, for any
$\lam,\mu\in\Lambda(m)$ and any $d\in\mathcal{D}_{\lam,\mu}$, can
be uniquely extends to a $K$-algebra map. To this end, it is enough
to show that for any $\lam,\mu,\nu\in\Lambda(m)$ and any
$d\in\mathcal{D}_{\lam,\mu}, d'\in\mathcal{D}_{\mu,\nu}$,
\addtocounter{equation}{6}
\begin{equation}\label{equa2}
\widehat{\phi}_{\lam,\mu}^{\widehat{d}}\widehat{\phi}_{\mu,\nu}^{\widehat{d}'}=\widehat{
{\phi}_{\lam,\mu}^{d}\phi_{\mu,\nu}^{d'}}.\end{equation}
Suppose that ${\phi}_{\lam,\mu}^{d}\phi_{\mu,\nu}^{d'}=\sum_{d''\in\mathcal{D}_{\lam,\nu}}
A_{d''}{\phi}_{\lam,\nu}^{d''}$, where $A_{d''}\in K$ for each $d''$. Then
$$
\widehat{{\phi}_{\lam,\mu}^{d}\phi_{\mu,\nu}^{d'}}=\sum_{d''\in\mathcal{D}_{\lam,\nu}}
A_{d''}\widehat{\phi}_{\lam,\nu}^{\widehat{d''}}.$$

By definition, it is easy to verify that for any $h\in\HH_q(\BS_m)$,
$$
\widehat{\phi}_{\mu,\nu}^{\widehat{d}'}(\widehat{x}_{\nu}\widehat{h})=\widehat{\phi_{\mu,\nu}^{{d}'}(x_{\nu}h)},\quad
\widehat{\phi}_{\lam,\mu}^{\widehat{d}}(\widehat{x}_{\mu}\widehat{h})=\widehat{\phi_{\lam,\mu}^{{d}}(x_{\mu}h)}.
$$
Therefore, $$
\widehat{\phi}_{\lam,\mu}^{\widehat{d}}\widehat{\phi}_{\mu,\nu}^{\widehat{d}'}(\widehat{x}_{\nu})=\widehat{
\phi_{\lam,\mu}^{{d}}\phi_{\mu,\nu}^{{d}'}(x_{\nu})}=\sum_{d''\in\mathcal{D}_{\lam,\nu}}A_{d''}\widehat{{\phi}_{\lam,\nu}^{d''}(x_{\nu})}=
\widehat{{\phi}_{\lam,\mu}^{d}\phi_{\mu,\nu}^{d'}}(\widehat{x}_{\nu}).
$$
This proves (\ref{equa2}), as required. The proof of the last two
equalities are straightforward and will be omitted.
\end{proof}

\addtocounter{theorem}{1}
\begin{definition} Let $\widetilde{A}(m)$ denote the non-unital $K$-subalgebra of $S_q(n)$ generated by
$S_q(m)\otimes\widehat{S}_q(m)$ and $H(m)$.
\end{definition}

The algebra $\widetilde{A}(m)$ can be regarded as a covering of the
algebra $A(m)$. Note that the unit element of $\widetilde{A}(m)$ is
the unit element of $S_q(m)\otimes \widehat{S}_q(m)$, which is
different from the unit element of $S_q(n)$. Note also that, by
Lemma \ref{lm36}, although $H(m)$ is not invertible in $S_q(n)$,
$H(m)$ is indeed an invertible element in the algebra
$\widetilde{A}(m)$.

\begin{definition} For each partition $\lam\in\Lambda^{+}(m)$, we define $$
Z_{\lam}:=y_{\lam'}T_{w_{\lam'}}\phi_{\lam,\omega_{m}}^{1}.
$$
\end{definition}

By \cite{DJ5}, the right ideal of $S_q(m)$ generated by $Z_{\lam}$ is the Weyl module $\Delta_{\lam}$ of $S_q(m)$ with
highest weight $\lam$.
In a similar way, we can define the element $\widehat{Z}_{\lam}$ and the Weyl module $\widehat{\Delta}_{\lam}$ of
$\widehat{S}_q(m)$.

\begin{lemma} Let $\lam,\mu\in\Lambda^{+}(m)$. Then we have $$
\bigl(z'_{\lam}\otimes\widehat{z'_{\mu}}\bigr)z_{m,m}=f_{(\lam,\mu)}(q)\bigl(z'_{\lam}\otimes\widehat{z'_{\mu}}\bigr),
\quad
\bigl(Z_{\lam}\otimes\widehat{Z}_{\mu}\bigr)Z_{m,m}=f_{(\lam,\mu)}(q)\bigl(Z_{\lam}\otimes\widehat{Z}_{\mu}\bigr).
$$
\end{lemma}
\begin{proof} Since the dual Specht module ${S}_{\lam}$ is also isomorphic to the right ideal of $\HH_q(\BS_m)$
generated by
$x_{\lam'}^{\#}T_{w_{\lam'}}^{\#}y_{\lam}^{\#}$, the first equality follows from \cite[Lemma 3.2]{Hu1}. The second
equality follows from the first equality and the definitions of $Z_{m,m}, Z_{\lam}$ and $\widehat{Z}_{\lam}$.
\end{proof}

To simplify notations, we shall denote $H(m)$ by $\theta$ and $S_q(m)\otimes \widehat{S}_q(m)$ by $S_q^{m,m}$. The set
of subspaces
$\{S_q^{m,m}, \theta S_q^{m,m}\}$ is a $\Z/2\Z$-Clifford system in $\widetilde{A}(m)$ in the sense of
\cite[(11.12)]{CR}. For any $\lam,\mu\in\Lambda^{+}(m)$, we set $$
\Delta_{\lam,\mu}:=\Delta_{\lam}\otimes\widehat{\Delta}_{\mu},\quad L_{\lam,\mu}:=L_{\lam}\otimes\widehat{L}_{\mu},
$$
where $\Delta_{\lam}, L_{\lam}$ (resp., $\widehat{\Delta}_{\mu},
\widehat{L}_{\mu}$) are the Weyl module, irreducible module of the
$q$-Schur algebra $S_q(m)$ (resp., $\widehat{S}_q(m)$) respectively.
By definition, $L_{\lam}$ (resp., $\widehat{L}_{\mu}$) is the unique
simple $S_q(m)$-head (resp., $\widehat{S}_q(m)$-head) of
$\Delta_{\lam}$ (resp., of $\widehat{\Delta}_{\mu}$). Hence
$L_{\lam,\mu}$ is the unique simple $S_q^{m,m}$-head of
$\Delta_{\lam,\mu}$.\smallskip

Note that by Lemma \ref{lm37}, the automorphism
\,\,\,$\widehat{\null}$\,\,\, of $S_q^{m,m}$ is induced by $\theta$. For any $S_q^{m,m}$-module
$M$, let $M^{\theta}$ be the new $S_q^{m,m}$-module obtained by
twisting the action of $S_q^{m,m}$ by \,\,\,$\widehat{\null}$\,\,\,.
We have the following result.

\begin{lemma} For any $\lam,\mu\in\Lambda^{+}(m)$, we have that $$
\bigl(\Delta_{\lam,\mu}\bigr)^{\theta}\cong \Delta_{\mu,\lam}, \quad \bigl(L_{\lam,\mu}\bigr)^{\theta}\cong L_{\mu,\lam}.
$$
\end{lemma}
\begin{proof} This follows directly from Lemma \ref{lm37}.
\end{proof}

\begin{definition} For any $\lam,\mu\in\Lambda^{+}(m)$, we set $$
\widetilde{\Delta}_{\lam,\mu}:=\Delta_{\lam,\mu}\!\uparrow\!_{S_q^{m,m}}^{\widetilde{A}(m)},\quad
\widetilde{L}_{\lam,\mu}:=L_{\lam,\mu}\!\uparrow\!_{S_q^{m,m}}^{\widetilde{A}(m)}.
$$
\end{definition}

Let $\widetilde{\sigma}$ be the automorphism of $\widetilde{A}(m)$
which is defined on generators by
$$ {\theta}^jx\mapsto (-1)^j{\theta}^j x,\,\,\forall\,x\in S_q^{m,m},\, j\in\mathbb{Z}.
$$
Clearly, $\widetilde{\sigma}\!\downarrow_{S_q^{m,m}}=\id$. By Lemma
\ref{lm36}, we can apply \cite[(2.2)]{Ge2} and
\cite[Appendix]{Hu6}. That is, as
$\widetilde{A}(m)$-$\widetilde{A}(m)$-bimodule,
\addtocounter{equation}{5}
\begin{equation}\label{equa3}
\widetilde{A}(m)\otimes_{S_q^{m,m}}\widetilde{A}(m)\cong
\widetilde{A}(m)\oplus
\bigl(\widetilde{A}(m)\bigr)^{\widetilde{\sigma}},
\end{equation}
where the left $\widetilde{A}(m)$-module structure on
$\bigl(\widetilde{A}(m)\bigr)^{\widetilde{\sigma}}$ was just given
by left multiplication, while the right $\widetilde{A}(m)$-module
structure on $\bigl(\widetilde{A}(m)\bigr)^{\widetilde{\sigma}}$ was
given by right multiplication twisted by
$\widetilde{\sigma}$.

\addtocounter{theorem}{1}
\begin{lemma} Let $\lam,\mu\in\Lambda^{+}(m)$. We have that \begin{enumerate}
\item if $\lam\neq\mu$, then $\widetilde{L}_{\lam,\mu}\cong\widetilde{L}_{\mu,\lam}$ is
a simple $\widetilde{A}(m)$-module;
\item there is a direct sum decomposition of $\widetilde{A}(m)$-module: $\widetilde{L}_{\lam,\lam}=
\widetilde{L}^{+}_{\lam,\lam}\oplus\widetilde{L}^{-}_{\lam,\lam}$,
where $$\begin{aligned}
\widetilde{L}^{+}_{\lam,\lam}:&=\bigl(\sqrt{f_{(\lam,\lam)}(q)}\overline{Z}_{\lam}\otimes\widehat{\overline{Z}}_{\lam}+
\overline{Z}_{\lam}\otimes\widehat{\overline{Z}}_{\lam}\theta\bigr)S_q^{m,m},\\
\widetilde{L}^{-}_{\lam,\lam}:&=\bigl(\sqrt{f_{(\lam,\lam)}(q)}\overline{Z}_{\lam}\otimes\widehat{\overline{Z}}_{\lam}-
\overline{Z}_{\lam}\otimes\widehat{\overline{Z}}_{\lam}\theta\bigr)S_q^{m,m},
\end{aligned}
$$
where $\overline{Z}_{\lam}$ (resp., $\widehat{\overline{Z}}_{\lam}$) is the natural image of $Z_{\lam}$ (resp., of
$\widehat{Z}_{\lam}$) in
$L_{\lam}$ (resp., $\widehat{L}_{\lam}$);
\item $\widetilde{A}(m)$ is split over $K$, and the set $$
\Bigl\{\widetilde{L}_{\lam,\mu},\widetilde{L}^{+}_{\beta,\beta},\widetilde{L}^{-}_{\beta,\beta}\Bigm|\lam,\mu,
\beta\vdash m, \lam\neq\mu, (\lam,\mu)\in\PP_n/\!\sim\Bigr\}
$$
forms a complete set of pairwise non-isomorphic absolutely simple $\widetilde{A}(m)$-modules;
\item $\bigl(\widetilde{L}_{\lam,\mu}\bigr)^{\widetilde{\sigma}}\cong\widetilde{L}_{\lam,\mu}\cong
\widetilde{L}_{\mu,\lam}$,
$\bigl(\widetilde{L}^{+}_{\beta,\beta}\bigr)^{\widetilde{\sigma}}\cong\widetilde{L}^{-}_{\beta,\beta}$.
\end{enumerate}
\end{lemma}
\begin{proof} We only give the proof of (4), as the other claims follow
from (4), (\ref{equa3}) and Frobenius reciprocity.

In fact, one can check that the following map gives the isomorphism
$\bigl(\widetilde{L}_{\lam,\mu}\bigr)^{\widetilde{\sigma}}\cong\widetilde{L}_{\lam,\mu}$: for any $h_1,
h_2\in S_q^{m,m}$,
$$\bigl(\overline{Z}_{\lam}\otimes\widehat{\overline{Z}}_{\mu}\bigr)h_1+
\bigl(\overline{Z}_{\lam}\otimes\widehat{\overline{Z}}_{\mu}\bigr)h_2\theta
\mapsto
-f_{(\lam,\mu)}(q)\bigl(\overline{Z}_{\lam}\otimes\widehat{\overline{Z}}_{\mu}\bigr)h_2+
\bigl(\overline{Z}_{\lam}\otimes\widehat{\overline{Z}}_{\mu}\bigr)h_1\theta;
$$
while the following map gives the isomorphism
$\bigl(\widetilde{L}^{+}_{\lam,\lam}\bigr)^{\widetilde{\sigma}}\cong\widetilde{L}^{-}_{\lam,\lam}$:
for any $h\in S_q^{m,m}$,
$$\bigl(\sqrt{f_{(\lam,\lam)}(q)}\overline{Z}_{\lam}\otimes\widehat{\overline{Z}}_{\lam}+
\overline{Z}_{\lam}\otimes\widehat{\overline{Z}}_{\lam}\theta\bigr)h
\mapsto\bigl(\sqrt{f_{(\lam,\lam)}(q)}\overline{Z}_{\lam}\otimes\widehat{\overline{Z}}_{\lam}-
\overline{Z}_{\lam}\otimes\widehat{\overline{Z}}_{\lam}\theta\bigr)h.
$$
\end{proof}

There is also a direct sum decomposition of $\widetilde{A}(m)$-module: $\widetilde{\Delta}_{\lam,\lam}=
\widetilde{\Delta}^{+}_{\lam,\lam}\oplus\widetilde{\Delta}^{-}_{\lam,\lam}$,
where $$\begin{aligned}
\widetilde{\Delta}^{+}_{\lam,\lam}:&=\bigl(\sqrt{f_{(\lam,\lam)}(q)}{Z}_{\lam}\otimes{\widehat{Z}_{\lam}}+
{Z}_{\lam}\otimes{\widehat{Z}_{\lam}}\theta\bigr)S_q^{m,m},\\
\widetilde{\Delta}^{-}_{\lam,\lam}:&=\bigl(\sqrt{f_{(\lam,\lam)}(q)}{Z}_{\lam}\otimes{\widehat{Z}_{\lam}}-
{Z}_{\lam}\otimes{\widehat{Z}_{\lam}}\theta\bigr)S_q^{m,m}.
\end{aligned}
$$

\begin{lemma} With the same notations as above, we have \begin{enumerate}
\item if $\lam\neq\mu$, then $\widetilde{L}_{\lam,\mu}$ is the unique simple $\widetilde{A}(m)$-head of
$\widetilde{\Delta}_{\lam,\mu}$;
\item $\widetilde{L}^{+}_{\lam,\lam}$ (resp., $\widetilde{L}^{-}_{\lam,\lam}$) is the unique simple
$\widetilde{A}(m)$-head of
$\widetilde{\Delta}^{+}_{\lam,\lam}$ (resp., of $\widetilde{\Delta}^{-}_{\lam,\lam}$);
\item $\bigl(\widetilde{\Delta}_{\lam,\mu}\bigr)^{\widetilde{\sigma}}\cong\widetilde{\Delta}_{\lam,\mu}\cong
\widetilde{\Delta}_{\mu,\lam}$,
$\bigl(\widetilde{\Delta}^{+}_{\lam,\lam}\bigr)^{\widetilde{\sigma}}\cong\widetilde{\Delta}^{-}_{\lam,\lam}$.
\end{enumerate}
\end{lemma}
\begin{proof} These follow from Frobenius reciprocity.
\end{proof}

\begin{corollary} \label{cor317} Let $\lam$ be a partition in $\Lambda^{+}(m)$. Let $\theta$ acts as the scalar
$\sqrt{f_{(\lam,\lam)}(q)}$ (resp., $-\sqrt{f_{(\lam,\lam)}(q)}$)
on the highest weight vector of $\Delta_{\lam,\lam}$. Then this action can be uniquely extends to a representation of
$\widetilde{A}(m)$ on $\Delta_{\lam,\lam}$. The resulting $\widetilde{A}(m)$-module is isomorphic to
$\Delta^{+}_{\lam,\lam}$ (resp., $\Delta^{-}_{\lam,\lam}$).
The same statements hold for $L^{+}_{\lam,\lam}$ and $L^{-}_{\lam,\lam}$.
\end{corollary}

Note that the set $\{\phi_{\lam,\lam}^{1}\}_{\lam\in\Lambda(m)}$
is a set of pairwise orthogonal idempotents in $S_q(m)$, and
$\sum_{\lam\in\Lambda(m)}\phi_{\lam,\lam}^{1}=1$. For any right
$S_q(m)$-module $M$, it is clear that
$$ M=\bigoplus_{\lam\in\Lambda(m)}M\phi_{\lam,\lam}^{1}.
$$
For each $\lam\in\Lambda(m)$, we define
$M_{\lam}:=M\phi_{\lam,\lam}^{1}$, and we call $M_{\lam}$ the
$\lam$-weight space of $S_q(m)$-module $M$. Note also that $S_q(m)$
is an epimorphic image of the quantum algebra
$U_{K}(\mathfrak{gl}_m)$ associated to $\mathfrak{gl}_m$ (cf.
\cite{BLM}, \cite{Du}). Any $S_q(m)$-module $M$ naturally becomes a
module over $U_{K}(\mathfrak{gl}_m)$. The definition of weight space
we used here coincides with the usual definition of weight space for
$U_{K}(\mathfrak{gl}_m)$-module. In a similar way, we can define the
weight space for any $\widehat{S}_q(m)$-module. Therefore, we have
also the notion of weight space for any $S_q^{m,m}$-module. The
weights of any $S_q^{m,m}$-module are elements in the set
$\underline{\Lambda}(m):=\Lambda(m)\times\Lambda(m)$.\smallskip

There is a natural additive group structure on $\underline{\Lambda}(m)$. Let
$\bigl\{e^{\ulam}\bigr\}_{\ulam\in\underline{\Lambda}(m)}$ denote the standard basis of the group ring
$\Z[\underline{\Lambda}(m)]$ over $\Z$. Then $e^{\ulam}e^{\umu}=e^{\ulam+\umu}$. For any finite dimensional
$S_q^{m,m}$-module $M$, we define the formal character of $M$ as $$
\chf M=\sum_{\ulam\in\underline{\Lambda}(m)}\dim M_{\ulam}\,e^{\ulam}\in\Z[\underline{\Lambda}(m)].
$$
For any short exact sequence $0\rightarrow M'\rightarrow M\rightarrow M''\rightarrow 0$ of finite dimensional
$S_q^{m,m}$-modules, it is clear that $$
\chf M=\chf M'+\chf M''.
$$
Therefore, the map $\chf$ is a map defined on the Grothendieck group $\mathcal{R}(S_q^{m,m})$
associated to the category of finite dimensional $S_q^{m,m}$-modules.\smallskip

Set $e=\phi_{\omega_m,\omega_m}^{1}\otimes\widehat{\phi}_{\omega_m,\omega_m}^{1}$. Then $e$ is an idempotent in
$\widetilde{A}(m)$, and $e\widetilde{A}(m)e=A(m), eS_q^{m,m}e=\HH_q(\BS_{(m,m)})$. We define a functor $\widetilde{F}$
from the category of finite dimensional $\widetilde{A}(m)$-modules to the category of finite dimensional $A(m)$-modules
as follows: for any finite dimensional $\widetilde{A}(m)$-module
$M, N$, and any $\varphi\in\Hom_{\widetilde{A}(m)}(M,N)$, $\widetilde{F}(M)=Me$, $\widetilde{F}(N)=Ne$, and $$
\widetilde{F}(\varphi)(xe):=\varphi(x)e,\quad\,\forall\,x\in M.
$$
Let ${F}$ be the Schur functor (induced by $e$) from the category of
finite dimensional $S_q^{m,m}$-modules to the category of finite
dimensional $\HH_q(\BS_{(m,m)})$-modules. Then we have the following
commutative diagram of functors: \addtocounter{equation}{3}
\begin{equation}\label{Res}
\begin{CD}
\Mod\text{-}\widetilde{A}(m)@>{\Res}>>\Mod\text{-}S_q^{m,m}\\
@V{\widetilde{F}} VV @V{F}VV\\
\Mod\text{-}A(m)@>{\Res}>>\Mod\text{-}\HH_q(\BS_{(m,m)})
\end{CD}.
\end{equation}
We set
$\underline{\Lambda}^{+}(m):=\Lambda^+(m)\times\Lambda^{+}(m)$.

\addtocounter{theorem}{1}
\begin{lemma} \label{schur1} Let $(\lam,\mu)\in\underline{\Lambda}^{+}(m)$ such that $\lam\neq\mu$. Then we
have $$\begin{aligned}
&\widetilde{F}\bigl(\widetilde{\Delta}_{\lam,\lam}^{\pm}\bigr)=S(\lam,\lam)_{\pm},\quad
\widetilde{F}\bigl(\widetilde{\Delta}_{\lam,\mu}\bigr)=S(\lam,\mu),\\
&\widetilde{F}\bigl(\widetilde{L}_{\lam,\lam}^{\pm}\bigr)=\begin{cases}
D(\lam,\lam)_{\pm},&\text{if $\lam$ is $e$-restricted;}\\
0,
&\text{otherwise,}\end{cases},\\
&\widetilde{F}\bigl(\widetilde{L}_{\lam,\mu}\bigr)=\begin{cases}
D(\lam,\mu),&\text{if $(\lam,\mu)$ is $e$-restricted;}\\
0, &\text{otherwise,}\end{cases}.
\end{aligned}$$
\end{lemma}
\begin{proof} These follow from (\ref{Res}) and direct verification.
\end{proof}

\begin{lemma} \label{schur2} $\widetilde{F}$ is an exact functor. In particular,
$\widetilde{F}$ induces a homomorphism from the Grothendieck group
$\mathcal{R}(\widetilde{A}(m))$ associated to the category of finite
dimensional $\widetilde{A}(m)$-modules to the Grothendieck group
$\mathcal{R}({A}(m))$ associated to the category of finite
dimensional ${A}(m)$-modules.
\end{lemma}
\begin{proof} This follows from the same arguments as in \cite[Section 6]{Gr}.
\end{proof}

\begin{corollary} \label{cor321} For any $\lam,\mu\in\Lambda^{+}(m)$ with $\mu$ being $e$-restricted, we have the following equality of decomposition numbers: $$
\bigl[\widetilde{\Delta}_{\lam,\lam}^{+}:\widetilde{L}_{\mu,\mu}^{+}\bigr]=\bigl[S(\lam,\lam)_{+}:D(\mu,\mu)_{+}\bigr]=
\bigl[S_{(\lam,\lam)}^{+}:D_{(\mu,\mu)}^{+}\bigr].
$$
\end{corollary}
\begin{proof} This follows directly from Lemma \ref{schur1} and \ref{schur2}.
\end{proof}

Therefore, computing the decomposition number $\bigl[S_{(\lam,\lam)}^{+}:D_{(\mu,\mu)}^{+}\bigr]$ can be reduced to
computing the decomposition number
$\bigl[\widetilde{\Delta}_{\lam,\lam}^{+}:\widetilde{L}_{\mu,\mu}^{+}\bigr]$. The latter will be done in the final
section.

\bigskip\bigskip
\section{Computing the Laurent polynomial $f_{\ulam}(v)$}

Let $a$ be a fixed integer with $0\leq a\leq n$ and $\ulam=(\lam^{(1)},\lam^{(2)})$ be a fixed $a$-bipartition of $n$. The
purpose of this section is to give a close formula for the Laurent polynomial $f_{\ulam}(v)$, which was introduced in
\cite[Lemma
3.2]{Hu1}. \smallskip

Recall that by \cite[Lemma 3.2]{Hu1}, certain central element (denoted by $z_{a,n-a}$ in
\cite[Lemma 3.2]{Hu1}) of $\HH_v(\BS_{(a,n-a)})$ acts on the dual Specht module $S_{\ulam}$ as the scalar
$f_{\ulam}(v)$. These central elements $z_{a,n-a}$ arise in the study of certain homomorphism between some (dual)
Specht modules over the Hecke algebra $\HH_v(B_n)$. Indeed, similar central elements do arise if we consider the more
general type $B_n$ Hecke algebra $\HH_{v,\tilde{v}}(B_n)$ which has two parameters $v,\widetilde{v}$, where
$\widetilde{v}$ is another indeterminate over $\Z$. By some abuse of notations, we will denote the resulting Laurent
polynomial by $f_{\ulam}(v,\widetilde{v})$. The relation with our previous introduced Laurent polynomial $f_{\ulam}(v)$
is given by $$
f_{\ulam}(v)=f_{\ulam}(v,1).$$
In this section, we shall give a closed formula for the
Laurent polynomial $f_{\ulam}(v,\widetilde{v})$.\smallskip

We fix some notations. Let $v, \tilde{v}$ be two indeterminates over $\Z$. Let
$\widetilde{\mathcal{A}}:=\Z[v,v^{-1},\tilde{v},\tilde{v}^{-1}]$.
The Hecke algebra $\HH_{v,\tilde{v}}(B_n)$ of type $B_n$ over $\widetilde{\mathcal{A}}$ is the
associative unital $\widetilde{\mathcal{A}}$-algebra with generators $T_0, T_1,\cdots,
T_{n-1}$ subject to the following relations
$$\begin{aligned}
&(T_0+1)(T_0-\tilde{v})=0,\\
&T_0T_1T_0T_1=T_1T_0T_1T_0,\\
&(T_i+1)(T_i-v)=0,\quad\text{for $1\leq i\leq n-1$,}\\
&T_iT_{i+1}T_i=T_{i+1}T_{i}T_{i+1},\quad\text{for $1\leq i\leq n-2$,}\\
&T_iT_j=T_jT_i,\quad\text{for $0\leq i<j-1\leq n-2$.}\end{aligned}
$$

Let $\tr$ be the trace form on the Hecke algebras
$\HH_{v,\tilde{v}}(B_n)$ and $\HH_v(\BS_n)$ as defined in
\cite[p. 697, line 30]{Ma1}. Note that the trace form is denoted by
$\tau$ in \cite[p. 697, line 30]{Ma1}.

\begin{definition} {\rm (\cite[(2.1)]{DJMu}, \cite[(3.8)]{DJ2})} \label{df41}
For any non-negative integers $k,a,b$, we set $$\begin{aligned}
u_{k}^{+}&=\begin{cases} 1 &\text{{\rm if}\,\,\, $k=0$,}\\
\prod\limits_{i=1}^{k}\Bigl(v^{i-1}+T_{i-1}\cdots
T_{1}T_{0}T_{1}\cdots T_{i-1}\Bigr) & \text{{\rm if}\,\,\, $1\leq
k\leq n$.} \end{cases} \\
u_{k}^{-}&=\begin{cases} 1 &\text{{\rm if}\,\,\, $k=0$,}\\
\prod\limits_{i=1}^{k}\Bigl(\tilde{v}v^{i-1}-T_{i-1}\cdots
T_{1}T_{0}T_{1}\cdots T_{i-1}\Bigr) & \text{{\rm if}\,\,\, $1\leq
k\leq n$.} \end{cases}\\
h_{a,b}&=T_{w_{a,b}}\end{aligned} $$ where $$
w_{a,b}=\begin{cases} \underbrace{s_{a}\cdots
s_{1}}\underbrace{s_{a+1}\cdots s_{2}}\cdots
\underbrace{s_{a+b-1}\cdots s_{b}}, &\text{if $a,b$ are positive integers,}\\
1. &\text{if $a$ or $b$ is zero.}\\
\end{cases} $$
\end{definition}

By some abuse of notations, for each integer $i\in\{1,2,\cdots,n-1\}\setminus\{a\}$, we set
$$
\widehat{i}=\begin{cases} {i+n-a} &\text{{\rm if}\,\,\, $1\leq i<a$,}\\
{i-a} &\text{{\rm if}\,\,\, $a<i\leq n-1$.}\end{cases} $$
Let \,\,\,$\widehat{\null}$\,\,\,\, be the group isomorphism from the Young subgroup
$\BS_{(a,n-a)}$ onto the Young subgroup $\BS_{(n-a,a)}$ which is defined on
generators by $\widehat{s_i}=s_{\widehat{i}}$ for each integer $i\in\{1,2,\cdots,n-1\}\setminus\{a\}$.
We use the same notation \,\,\,$\widehat{\null}$\,\,\, to denote the algebra
isomorphism from the Hecke algebra $\HH_v(\BS_{(a,n-a)})$ onto the Hecke algebra $\HH_v(\BS_{(n-a,a)})$
which is defined on generators by $\widehat{T_i}=T_{\widehat{i}}$
for each integer $i\in\{1,2,\cdots,n-1\}\setminus\{a\}$.
By abuse of notation, the inverse of \,\,\,$\widehat{\null}$\,\,\, will
also be denoted by \,\,$\widehat{\null}$\,\,\,\,. By
\cite[(3.23)]{DJ3}, we know that there exists an element
$z_{a,n-a}$ in the center of $\HH_v(\BS_{(a,n-a)})\otimes_{\mathcal{A}}\widetilde{\mathcal{A}}$ such that
\addtocounter{equation}{1}
\begin{equation}\label{eq42}
u_{n-a}^{-}h_{n-a,a}u_{a}^{+}h_{a,n-a}u_{n-a}^{-}h_{n-a,a}u_{a}^{+}=
u_{n-a}^{-}h_{n-a,a}u_{a}^{+}z_{a,n-a}. \end{equation}

Let $\ast$ be the anti-automorphism of $\HH_{v,\tilde{v}}(B_n)$ which is defined on generators by $T_i^{\ast}=T_i$ for
each
$0\leq i\leq n-1$.

\addtocounter{theorem}{1}
\begin{lemma}\label{lm43} For any integer $a$ with $1\leq a\leq n$, we have that
\begin{enumerate}
\item[(4.3.1)] $u_{a}^{+}h_{a,n-a}u_{n-a}^{-}h_{n-a,a}u_{a}^{+}h_{a,n-a}u_{n-a}^{-}=
u_{a}^{+}h_{a,n-a}u_{n-a}^{-}\hat{z}_{a,n-a}$,
\item[(4.3.2)] $\bigl(z_{a,n-a}\bigr)^{\ast}=z_{a,n-a}$.
\end{enumerate}
\end{lemma}

\begin{proof} Using the same argument as \cite[(3.23)]{DJ3}, we know that there exists an element
$z'_{n-a,a}$ in the center of $\HH_q(\BS_{(n-a,a)})\otimes_{\mathcal{A}}\widetilde{\mathcal{A}}$ such that $$
u_{a}^{+}h_{a,n-a}u_{n-a}^{-}h_{n-a,a}u_{a}^{+}h_{a,n-a}u_{n-a}^{-}=
u_{a}^{+}h_{a,n-a}u_{n-a}^{-}z'_{n-a,a}. $$

Note that $$\begin{aligned}
&\quad\,\,u_{n-a}^{-}h_{n-a,a}u_{a}^{+}h_{a,n-a}u_{n-a}^{-}h_{n-a,a}u_{a}^{+}h_{a,n-a}
u_{n-a}^{-}h_{n-a,a}u_{a}^{+}\\
&=u_{n-a}^{-}h_{n-a,a}u_{a}^{+}h_{a,n-a}u_{n-a}^{-}h_{n-a,a}u_{a}^{+}z_{a,n-a}\\
&=u_{n-a}^{-}h_{n-a,a}u_{a}^{+}z_{a,n-a}^2. \end{aligned} $$ On
the other hand, we have that $$
\begin{aligned}
&\quad\,\,u_{n-a}^{-}h_{n-a,a}u_{a}^{+}h_{a,n-a}u_{n-a}^{-}h_{n-a,a}u_{a}^{+}h_{a,n-a}
u_{n-a}^{-}h_{n-a,a}u_{a}^{+}\\
&=u_{n-a}^{-}h_{n-a,a}u_{a}^{+}h_{a,n-a}u_{n-a}^{-}z'_{n-a,a}h_{n-a,a}u_{a}^{+}\\
&=u_{n-a}^{-}h_{n-a,a}u_{a}^{+}h_{a,n-a}u_{n-a}^{-}z'_{n-a,a}h_{n-a,a}u_{a}^{+}\hat{z'}_{n-a,a}\\
&=u_{n-a}^{-}h_{n-a,a}u_{a}^{+}z_{a,n-a}\hat{z'}_{n-a,a}.
\end{aligned} $$
It follows that $z_{a,n-a}^2=z_{a,n-a}\hat{z'}_{n-a,a}$. Note that
(see \cite{DJ3}) $z_{a,n-a}$ is invertible in
${\HH}_{v,\tilde{v}}(\BS_n)\otimes_{\widetilde{\mathcal{A}}}\mathbb{Q}(v,\tilde{v})$. Hence we conclude that
$z'_{n-a,a}=\hat{z}_{a,n-a}$, as required. This proves 1).

Applying the anti-automorphism $\ast$ to both sides of
(\ref{eq42}), we get that $$\begin{aligned}
u_{a}^{+}h_{a,n-a}u_{n-a}^{-}\hat{z}_{a,n-a}&=u_{a}^{+}h_{a,n-a}u_{n-a}^{-}h_{n-a,a}u_{a}^{+}h_{a,n-a}u_{n-a}^{-}\\
&=z_{a,n-a}^{\ast}u_{a}^{+}h_{a,n-a}u_{n-a}^{-}=u_{a}^{+}h_{a,n-a}u_{n-a}^{-}\widehat{z_{a,n-a}^{\ast}},
\end{aligned}
$$
which implies that $\bigl(z_{a,n-a}\bigr)^{\ast}=z_{a,n-a}$, as
required.
\end{proof}

For any $a$-bipartition $\ulam=(\lam^{(1)},\lam^{(2)})$ of $n$, we
define $z_{\ulam}=z_{\lam^{(1)}}\widehat{z_{\lam^{(2)}}}$, where
$z_{\lam^{(1)}}=x_{\lam^{(1)}}T_{w_{\lam^{(1)}}}y_{\lam^{(1)'}}$,
$z_{\lam^{(2)}}=x_{\lam^{(2)}}T_{w_{\lam^{(2)}}}y_{\lam^{(2)'}}$. We set
$\ulam':=(\lam^{(1)'},\lam^{(2)'})$. Recall that
$\widehat{\ulam}=(\lam^{(2)},\lam^{(1)})$, which is an
$(n-a)$-bipartition of $n$. We define the Specht modules $S^{\ulam}$ and the twisted Specht module
$\widehat{S}^{\widehat\ulam}$
to be: $$
S^{\ulam}:=u_{n-a}^{-}h_{n-a,a}u_{a}^{+}z_{\ulam}\HH_{v,\tilde{v}}(B_n),\quad
\widehat{S}^{\widehat\ulam}:=u_{a}^{+}h_{a,n-a}u_{n-a}^{-}z_{\widehat\ulam}\HH_{v,\tilde{v}}(B_n).
$$
Let $\theta_{\ulam}$ (resp., $\delta_{\ulam}$) be the right
$\HH_{v,\tilde{v}}(B_n)$-module homomorphism from $S^{\ulam}$
to $\widehat{S}^{\widehat\ulam}$ (resp., from $\widehat{S}^{\widehat\ulam}$ to $S^{\ulam}$) given by left
multiplication with $u_{a}^{+}h_{a,n-a}$ (resp., with
$u_{n-a}^{-}h_{n-a,a}$). It is clear that both $\theta_{\ulam}$ and
$\delta_{\ulam}$ are well-defined right
$\HH_{v,\tilde{v}}(B_n)$-module homomorphisms.

\begin{lemma} \label{lm44} For any $a$-bipartition $\ulam$ of $n$, there
exists a Laurent polynomial $f_{\ulam}(v,\tilde{v})\in
\widetilde{\mathcal{A}}$ such that
$z_{a,n-a}z_{\ulam}=f_{\ulam}(v,\tilde{v})z_{\ulam},\,\,
\hat{z}_{n-a,a}z_{\widehat{\ulam}}=f_{\ulam}(v,\tilde{v})z_{\widehat\ulam}$.
In particular, both $\theta_{\ulam}\delta_{\ulam}$ and
$\delta_{\ulam}\theta_{\ulam}$ are scalar multiplication by
$f_{\ulam}(v,\tilde{v})$.
\end{lemma}

\begin{proof} The first part is an easy consequence of
\cite[(4.1)]{DJ1}, by using the same argument as in the proof of
\cite[(3.2)]{Hu1}. The second part follows from Lemma \ref{lm43}.
\end{proof}

Let $\mu$ be a partition of $a$. Let $z_{\mu}:=x_{\mu}T_{w_{\mu}}y_{\mu'}$. Then $S^{\mu}:=z_{\mu}\HH_v(\BS_a)$ is the
Specht module of $\HH_v(\BS_a)$ corresponding to $\mu$. By \cite[(3.5)]{DJ2} and \cite[(5.2), (5.3)]{Mu}, the dual
Specht module $S_{\mu}$ of $\HH_v(\BS_a)$ is isomorphic to the right ideal generated by $y_{\mu'}T_{w_{\mu'}}x_{\mu}$.
Let $\theta_{\mu}$ (resp., $\delta_{\mu}$) be the right $\HH_v(\BS_a)$-module homomorphism from $S^{\mu}$ to
$y_{\mu'}T_{w_{\mu'}}x_{\mu}\HH_v(\BS_a)$ (resp., from $y_{\mu'}T_{w_{\mu'}}x_{\mu}\HH_v(\BS_a)$ to $S^{\mu}$)
given by left multiplication with $y_{\mu'}T_{w_{\mu'}}$ (resp., with $x_{\mu}T_{w_{\mu}}$). The following result is
coming from \cite[(5.8)]{KM}. But we shall give a different proof here in order to illustrate the technique which will
be used in the proof of the main theorem in this section.

\begin{lemma} {\rm (\cite[(5.8)]{KM})} \label{KMthm} Let $\mu$ be a partition of $a$. Then both
$\theta_{\mu}\delta_{\mu}$ and $\delta_{\mu}\theta_{\mu}$ are
scale multiplication of $v^{\ell(w_{\mu})}s_{\mu}(v)$, where
$s_{\mu}(v)\in {\mathcal A}$ is the Schur element associated to $\mu$. In particular, $$
z_{\mu}T_{w_{\mu'}}z_{\mu}=v^{\ell(w_{\mu})}s_{\mu}(v)z_{\mu}.
$$
\end{lemma}

\begin{proof} By \cite[(4.1)]{DJ1}, we have that $$
x_{\mu}\bigl(T_{w_{\mu}}y_{\mu'}T_{w_{\mu'}}x_{\mu}T_{w_{\mu}}\bigr)y_{\mu'}
=r_{\mu}(v)x_{\mu}T_{w_{\mu}}y_{\mu'},
$$
for some $r_{\mu}(v)\in {\mathcal A}$. Therefore, for any $h\in\HH_v(\BS_a)$,
$$\begin{aligned}
\delta_{\mu}\theta_{\mu}(z_{\mu}h)&=x_{\mu}T_{w_{\mu}}y_{\mu'}T_{w_{\mu'}}z_{\mu}h
=z_{\mu}T_{w_{\mu'}}z_{\mu}h\\
&=x_{\mu}\bigl(T_{w_{\mu}}y_{\mu'}T_{w_{\mu'}}x_{\mu}T_{w_{\mu}}\bigr)y_{\mu'}h
=r_{\mu}(v)x_{\mu}T_{w_{\mu}}y_{\mu'}h=r_{\mu}(v)z_{\mu}h,\end{aligned}
$$

We claim that $r_{\mu}(v)\neq 0$. In fact, by \cite[(5.9)]{Ma1},
$$\begin{aligned}
\tr\bigl(\delta_{\mu}(y_{\mu'}T_{w_{\mu'}}x_{\mu})\bigr)&=\tr\bigl(x_{\mu}T_{w_{\mu}}y_{\mu'}T_{w_{\mu'}}x_{\mu}\bigr)
=\tr\bigl(x_{\mu}^2T_{w_{\mu}}y_{\mu'}T_{w_{\mu'}}\bigr)\\
&=\biggl(\sum_{w\in\BS_{\mu}}v^{\ell(w)}\biggr)\tr(x_{\mu}T_{w_{\mu}}y_{\mu'}T_{w_{\mu'}})
=v^{\ell(w_{\mu})}\sum_{w\in\BS_{\mu}}v^{\ell(w)}\neq 0.
\end{aligned}
$$
It follows that $\delta_{\mu}(y_{\mu'}T_{w_{\mu'}}x_{\mu})\neq 0$, hence
$(\delta_{\mu})_{\mathbb{Q}(v)}$ is an
$\HH_{v}(\BS_a)\otimes_{\mathcal A}\mathbb{Q}(v)$-module isomorphism between two simple
$\bigl(\HH_{v}(\BS_a)\otimes_{\mathcal A}\mathbb{Q}(v)\bigr)$-modules. By similar argument, one can
show that $(\theta_{\mu})_{\mathbb{Q}(v)}$ is an
$\bigl(\HH_{v}(\BS_a)\otimes_{\mathcal A}\mathbb{Q}(v)\bigr)$-module isomorphism between two simple
$\bigl(\HH_{v}(\BS_a)\otimes_{\mathcal A}\mathbb{Q}(v)\bigr)$-modules. In particular,
this implies that $r_{\mu}(v)\neq 0$, as required.

It remains to show $r_{\mu}(v)=v^{\ell(w_{\mu})}s_{\mu}(v)$.
Since $r_{\mu}(v)\neq 0$,
$r_{\mu}(v)^{-1}z_{\mu}T_{w_{\mu'}}$ is an idempotent in
$\HH_{v}(\BS_a)\otimes_{\mathcal A}\mathbb{Q}(v)$. Clearly
$$S^{\mu}\otimes_{\mathcal{A}}\mathbb{Q}(v)=r_{\mu}(v)^{-1}z_{\mu}T_{w_{\mu'}}\bigl(\HH_{v}(\BS_a)
\otimes_{\mathcal A}\mathbb{Q}(v)\bigr).$$
Applying \cite[(1.6),(5.9)]{Ma1}, we get that $$
s_{\mu}(v)=\frac{1}{\tr\Bigl(r_{\mu}(v)^{-1}z_{\mu}T_{w_{\mu'}}\Bigr)}=
\frac{r_{\mu}(v)}{\tr\Bigl(z_{\mu}T_{w_{\mu'}}\Bigr)}=
v^{-\ell(w_{\mu})}r_{\mu}(v).
$$
Clearly $\delta_{\mu}\theta_{\mu}=v^{\ell(w_{\mu})}s_{\mu}(v)$
implies that
$\theta_{\mu}\delta_{\mu}=v^{\ell(w_{\mu})}s_{\mu}(v)$. This
completes the proof of the theorem.
\end{proof}

\begin{theorem} \label{main2} With the above notations, we have that
$$f_{\ulam}(v,\tilde{v})=v^{\frac{n(n-1)}{2}}\tilde{v}^{n-a}\frac{s_{\ulam}(v,\tilde{v})}{s_{\lam^{(1)}}(v)
s_{\lam^{(2)}}(v)}\in\widetilde{\mathcal{A}},$$ where
$s_{\ulam}(v,\tilde{v})$, (resp., $s_{\lam^{(1)}}(v)$, $
s_{\lam^{(2)}}(v)$) is the Schur element associated to the
bipartition $\ulam$ (resp., the partitions $\lam^{(1)}$,
$\lam^{(2)}$). In particular, in the ring $\widetilde{\mathcal{A}}$,
$$s_{\lam^{(1)}}(v)s_{\lam^{(2)}}(v)\bigm|s_{\ulam}(v,\tilde{v}). $$
\end{theorem}

\begin{proof} For an $a$-bipartition $\ulam=(\lam^{(1)},\lam^{(2)})$, we define
$T_{w_{\ulam}}=T_{w_{\lam^{(1)}}}\widehat{T_{w_{\lam^{(2)}}}}$. Then
$T_{w_{\widehat{\ulam}}}=T_{w_{\lam^{(2)}}}\widehat{T_{w_{\lam^{(1)}}}}$. Now applying Lemma
\ref{KMthm}, we get that $$\begin{aligned}
z_{\widehat\ulam}T_{w_{\widehat{\ulam'}}}z_{\widehat\ulam}&=
z_{\lam^{(2)}}T_{w_{\lam^{(2)'}}}z_{\lam^{(2)}}\widehat{z_{\lam^{(1)}}}\widehat{T_{w_{\lam^{(1)'}}}}
\widehat{z_{\lam^{(1)}}}\\
&=v^{\ell(w_{\lam^{(2)}})+\ell(w_{\lam^{(1)}})}s_{\lam^{(2)}}(v)s_{\lam^{(1)}}(v)
z_{\lam^{(2)}}\widehat{z_{\lam^{(1)}}}\\
&=v^{\ell(w_{\lam^{(2)}})+\ell(w_{\lam^{(1)}})}s_{\lam^{(2)}}(v)s_{\lam^{(1)}}(v)
z_{\widehat\ulam},\end{aligned}
$$
where $s_{\lam^{(1)}}(v)$, $ s_{\lam^{(2)}}(v)$ are the Schur
elements associated to the partitions $\lam^{(1)}$, $\lam^{(2)}$
respectively.

Let
$A(v,\tilde{v}):=v^{-\ell(w_{\lam^{(2)}})-\ell(w_{\lam^{(1)}})}f_{\ulam}(v,\tilde{v})^{-1}s_{\lam^{(2)}}(v)^{-1}
s_{\lam^{(1)}}(v)^{-1}$. Then we have $$\begin{aligned}
&\quad\,\,\Bigl(A(v,\tilde{v})u_{n-a}^{-}h_{n-a,a}u_{a}^{+}z_{\ulam}h_{a,n-a}T_{w_{\widehat{\ulam'}}}\Bigr)^2\\
&=\Bigl(A(v,\tilde{v})u_{n-a}^{-}h_{n-a,a}u_{a}^{+}h_{a,n-a}z_{\widehat\ulam}T_{w_{\widehat{\ulam'}}}\Bigr)^2\\
&=A(v,\tilde{v})^2u_{n-a}^{-}h_{n-a,a}u_{a}^{+}h_{a,n-a}z_{\widehat\ulam}T_{w_{\widehat{\ulam'}}}
u_{n-a}^{-}h_{n-a,a}u_{a}^{+}h_{a,n-a}z_{\widehat\ulam}T_{w_{\widehat{\ulam'}}}\\
&=A(v,\tilde{v})^2u_{n-a}^{-}h_{n-a,a}u_{a}^{+}h_{a,n-a}
u_{n-a}^{-}h_{n-a,a}u_{a}^{+}h_{a,n-a}z_{\widehat\ulam}T_{w_{\widehat{\ulam'}}}z_{\widehat\ulam}T_{w_{\widehat{\ulam'}}}
\\
&=A(v,\tilde{v})f_{\ulam}(v,\tilde{v})^{-1}u_{n-a}^{-}h_{n-a,a}u_{a}^{+}h_{a,n-a}
u_{n-a}^{-}h_{n-a,a}u_{a}^{+}h_{a,n-a}z_{\widehat\ulam}T_{w_{\widehat{\ulam'}}}\\
&=A(v,\tilde{v})f_{\ulam}(v,\tilde{v})^{-1}u_{n-a}^{-}h_{n-a,a}u_{a}^{+}z_{a,n-a}h_{a,n-a}z_{\widehat\ulam}
T_{w_{\widehat{\ulam'}}}\\
&=A(v,\tilde{v})f_{\ulam}(v,\tilde{v})^{-1}u_{n-a}^{-}h_{n-a,a}u_{a}^{+}z_{a,n-a}z_{\ulam}h_{a,n-a}
T_{w_{\widehat{\ulam'}}}\\
&=A(v,\tilde{v})u_{n-a}^{-}h_{n-a,a}u_{a}^{+}z_{\ulam}h_{a,n-a}T_{w_{\widehat{\ulam'}}}.
\end{aligned}
$$
In other words,
$A(v,\tilde{v})u_{n-a}^{-}h_{n-a,a}u_{a}^{+}z_{\ulam}h_{a,n-a}T_{w_{\widehat{\ulam'}}}$
is an idempotent of the Hecke algebra
$\HH_{v,\tilde{v}}(B_n)\otimes_{\widetilde{\mathcal{A}}}\mathbb{Q}(v,\tilde{v})$. Since $$
S^{\ulam}\otimes_{\widetilde{\mathcal{A}}}\mathbb{Q}(v,\tilde{v})=
A(v,\tilde{v})u_{n-a}^{-}h_{n-a,a}u_{a}^{+}z_{\ulam}h_{a,n-a}T_{w_{\widehat{\ulam'}}}
\bigl(\HH_{v,\tilde{v}}(B_n)\otimes_{\widetilde{\mathcal{A}}}\mathbb{Q}(v,\tilde{v})\bigr),
$$
it follows from \cite[(1.6)]{Ma1} that $$
s_{\ulam}(v,\tilde{v})=\frac{1}{\tr\bigl(A(v,\tilde{v})u_{n-a}^{-}h_{n-a,a}u_{a}^{+}
z_{\ulam}h_{a,n-a}T_{w_{\widehat{\ulam'}}}\bigr)},
$$
where $s_{\ulam}(v,\tilde{v})$ is the Schur element associated to
the bipartition $\ulam$. On the other hand, if we set $q=v,
Q_1=-1, Q_2=\tilde{v}$, and $\mu=(\lam^{(2)},\lam^{(1)})$, then the element $m_{\mu}$ in \cite{Ma1}
is $$x_{\lam^{(2)}}\widehat{x_{\lam^{(1)}}}\biggl(\prod_{k=1}^{n-a}(-v^{1-k})\biggr)u_{n-a}^{-}$$ in our
notation, and the element $n_{\mu'}$ in \cite{Ma1} is
$$y_{\lam^{(1)'}}\widehat{y_{\lam^{(2)'}}}\biggl(\prod_{k=1}^{a}v^{1-k}\biggr))u_{a}^{+}$$ in our notation.
Note also that the element $w_{\mu'}$ in \cite[(5.9)]{Ma1} is just
$$w_{a,n-a}w_{\lam^{(2)'}}\widehat{w_{\lam^{(1)'}}}$$ in our notation.
Therefore, the element $z_{\mu}T_{w_{\mu'}}$ in the notation of \cite[(5.9)]{Ma1} is just $$
\Bigl(\prod_{k=1}^{n-a}(-v^{1-k})\Bigr)\Bigl(\prod_{k=1}^{a}v^{1-k}\Bigr)
u_{n-a}^{-}h_{n-a,a}u_{a}^{+}z_{\ulam}h_{a,n-a}T_{w_{\widehat{\ulam'}}}
$$
in our notations. Applying \cite[(5.9)]{Ma1}, we have
$$\begin{aligned}
&\quad\,\,\tr\biggl(\Bigl(\prod_{k=1}^{n-a}(-v^{1-k})\Bigr)\Bigl(\prod_{k=1}^{a}v^{1-k}\Bigr)
u_{n-a}^{-}h_{n-a,a}u_{a}^{+}z_{\ulam}h_{a,n-a}T_{w_{\widehat{\ulam'}}}\biggr)\\
&=(-1)^nv^{\ell(w_{a,n-a})+\ell(w_{\lam^{(2)}})
+\ell(w_{\lam^{(1)}})}(-1)^{a}\tilde{v}^{n-a}.\end{aligned}
$$
It follows that $$ \tr\Bigl(u_{n-a}^{-}h_{n-a,a}u_{a}^{+}
z_{\ulam}h_{a,n-a}T_{w_{\widehat{\ulam'}}}\Bigr)=v^{\frac{n(n-1)}{2}+\ell(w_{\lam^{(2)}})
+\ell(w_{\lam^{(1)}})}\tilde{v}^{n-a}.
$$
Hence, $$
f_{\ulam}(v,\tilde{v})=v^{\frac{n(n-1)}{2}}\tilde{v}^{n-a}
\frac{s_{\ulam}(v,\tilde{v})}{s_{\lam^{(1)}}(v)s_{\lam^{(2)}}(v)}\in\widetilde{\mathcal A},
$$
as required.
\end{proof}

\begin{corollary}  With the above notations, we have that
 $$
f_{\ulam}(v)=v^{\frac{n(n-1)}{2}}\frac{s_{\ulam}(v,1)}{s_{\lam^{(1)}}(v)s_{\lam^{(2)}}(v)}.
$$
\end{corollary}

\bigskip\bigskip\smallskip
\section{Twining character formulae}

The purpose of this final section is to derive a closed formula of the decomposition number
$\bigl[\widetilde{\Delta}_{\lam,\lam}^{+}:\widetilde{L}_{\mu,\mu}^{+}\bigr]$ (where $\lam,\mu\in\Lambda^{+}(m)$) in
terms of the elements $f_{(\lam,\lam)}(q), f_{(\mu,\mu)}(q)$, and thus proving our main result Theorem \ref{mainthm}.
\smallskip

Let $M$ be an arbitrary finite dimensional $\widetilde{A}(m)$-module. Then $M\theta\subseteq M$. For each
$\mu\in\Lambda(m)$, it follows from the definition that $$
\theta\phi_{\mu,\mu}^{1}=\widehat{\phi}_{\mu,\mu}^{1}\theta,\quad \phi_{\mu,\mu}^{1}\theta=
\theta\widehat{\phi}_{\mu,\mu}^{1}.
$$
Let $M_{(\mu,\mu)}$ be the $(\mu,\mu)$-weight space of the
$\widetilde{A}(m)$-module $M$. Note that for any $x\in M$, $x\in
M_{(\mu,\mu)}$ if and only
$x\phi_{\mu,\mu}^{1}=x=x\widehat{\phi}_{\mu,\mu}^{1}$. Then, it is
easy to verify that $M_{(\mu,\mu)}\theta\subseteq M_{(\mu,\mu)}$.
Therefore, it makes sense to talk about the trace of $\theta$ on
$M_{(\mu,\mu)}$.\smallskip

We define the twining character of the $\widetilde{A}(m)$-module $M$
as follows: $$ \chf^{\theta}(M):=\sum_{\mu\in\Lambda(m)}\Tr(\theta,
M_{(\mu,\mu)})e^{\mu}\in K[\Lambda(m)],
$$
where $K[\Lambda(m)]$ denotes the group algebra of the additive
group $\Lambda(m)$. It is easy to see that $\chf^{\theta}$ induces a
homomorphism from the Grothendieck group $\mathcal{R}(\widetilde{A}(m))$ to
$K[\Lambda(m)]$. We denote this homomorphism again by $\chf^{\theta}$. We use
$\pi$ to denote the natural map from the group ring $\Z[\Lambda(m)]$
to the group algebra $K[\Lambda(m)]$.\smallskip

\begin{lemma} \label{lm51} For any $\lam,\mu\in\Lambda^{+}(m)$ with $\lam\neq\mu$, we have that
$$\begin{aligned}
&\chf^{\theta}\bigl(\widetilde{\Delta}_{\lam,\lam}^{+}\bigr)=-\chf^{\theta}\bigl(\widetilde{\Delta}_{\lam,\lam}^{-}\bigr)
=
\sqrt{f_{(\lam,\lam)}(q)}\pi\bigl(\chf\Delta_{\lam}\bigr),\\
&\chf^{\theta}\bigl(\widetilde{L}_{\lam,\lam}^{+}\bigr)=-\chf^{\theta}\bigl(\widetilde{L}_{\lam,\lam}^{-}\bigr)=
\sqrt{f_{(\lam,\lam)}(q)}\pi\bigl(\chf
L_{\lam}\bigr),\\
&\chf^{\theta}\bigl(\widetilde{L}_{\lam,\mu}\bigr)=0,
\end{aligned}
$$ where $\chf\Delta_{\lam}, \chf L_{\lam}$ denote the formal
characters of the Weyl module and irreducible module associated to
$\lam$ of the $q$-Schur algebra $S_q(m)$ respectively.
\end{lemma}
\begin{proof} Let $\nu\in\Lambda(m)$. By definition, $\widetilde{L}_{\lam,\mu}=(L_{\lam}\otimes L_{\mu})\oplus
(L_{\lam}\otimes L_{\mu})\theta$, and
the $(\nu,\nu)$-weight space of $\widetilde{L}_{\lam,\mu}$ is just $$
((L_{\lam})_{\nu}\otimes (L_{\mu})_{\nu})\oplus ((L_{\lam})_{\nu}\otimes (L_{\mu})_{\nu})\theta.
$$
It is easy to see that the trace of $\theta$ on this space is $0$. This proves $\chf^{\theta}
\bigl(\widetilde{L}_{\lam,\mu}\bigr)=0$.

Let $v_{\lam}$ (resp., $\widehat{v}_{\lam}$) be the highest weight vector of $\Delta_{\lam}$ (resp., of
$\widehat{\Delta}_{\lam}$).
By definition, we know that the $(\nu,\nu)$-weight space of $\widetilde{\Delta}_{\lam,\lam}^{+}$ is
$(\Delta_{\lam})_{\nu}\otimes
(\widehat{\Delta}_{\lam})_{\nu}$. Let $\bigl\{v_{\lam}x_i\bigr\}_{i=1}^{k}$ be a $K$-basis of $(\Delta_{\lam})_{\nu}$,
then $\bigl\{\widehat{v}_{\lam}\widehat{x}_i\bigr\}_{i=1}^{k}$ is a $K$-basis of $(\widehat{\Delta}_{\lam})_{\nu}$. We
now apply Corollary \ref{cor317}.
For any integer $1\leq i, j\leq k$, we have $$
\bigl(v_{\lam}x_i\otimes\widehat{v}_{\lam}\widehat{x}_j\bigr)\frac{\theta}{\sqrt{f_{(\lam,\lam)}(q)}}=
v_{\lam}x_j\otimes\widehat{v}_{\lam}\widehat{x}_i.
$$
It follows that $$
\Tr\Bigl(\theta, \bigl(\widetilde{\Delta}_{\lam,\lam}^{+}\bigr)_{(\nu,\nu)}\Bigr)=\sqrt{f_{(\lam,\lam)}(q)}
\dim\bigl(\Delta_{\lam}\bigr)_{\nu}.
$$
Hence $\chf^{\theta}\bigl(\widetilde{\Delta}_{\lam,\lam}^{+}\bigr)=\sqrt{f_{(\lam,\lam)}(q)}\pi\bigl(\chf
\Delta_{\lam}\bigr)$. The remaining equalities can be proved in a similar way. This completes the proof of the lemma.
\end{proof}

For any $\lam,\mu,\nu\in\Lambda^{+}(m)$ with $\mu\neq\nu$, we set $$ d_{\lam,\mu}:=\bigl[\Delta_{\lam}:L_{\mu}\bigr],
\quad
m_{\lam,\mu}:=\bigl[\widetilde{\Delta}_{\lam,\lam}^{+}:\widetilde{L}_{\mu,\mu}^{+}\bigr],\quad
m_{\mu,\nu}^{\lam}:=\bigl[\widetilde{\Delta}_{\lam,\lam}^{+}:\widetilde{L}_{\mu,\nu}\bigr].
$$
By (\ref{Res}), it is clear that $$
\bigl[\widetilde{\Delta}_{\lam,\lam}^{+}:\widetilde{L}_{\mu,\mu}^{-}\bigr]=d_{\lam,\mu}^2-m_{\lam,\mu}.
$$
Therefore, we have the following equality in the Grothendieck group
$\mathcal{R}(\widetilde{A}(m))$: $$
\bigl[\widetilde{\Delta}_{\lam,\lam}^{+}\bigr]=\sum_{\mu\in\Lambda^{+}(m)}\Bigl(m_{\lam,\mu}
\bigl[\widetilde{L}_{\mu,\mu}^{+}\bigr]+(d_{\lam,\mu}^2-m_{\lam,\mu})
\bigl[\widetilde{L}_{\mu,\mu}^{-}\bigr]\Bigr)+\sum_{\substack{\mu,\nu\in\Lambda^{+}(m)\\
\mu\neq\nu}}m_{\mu,\nu}^{\lam}\bigl[\widetilde{L}_{\mu,\nu}\bigr].
$$
We apply the map $\chf^{\theta}$ to the above equality and using Lemma \ref{lm51}. It follows that
\addtocounter{equation}{1}
\begin{equation}\label{equa52}
\sqrt{f_{(\lam,\lam)}(q)}\pi\bigl(\chf\Delta_{\lam}\bigr)=\sum_{\mu\in\Lambda^{+}(m)}\bigl(2m_{\lam,\mu}-
d_{\lam,\mu}^2\bigr)
\sqrt{f_{(\mu,\mu)}(q)}\pi\bigl(\chf L_{\mu}\bigr).
\end{equation}
On the other hand, we have that (in $\Z[\Lambda^{+}(m)]$) $$
\chf\Delta_{\lam}=\sum_{\mu\in\Lambda^{+}(m)}d_{\lam,\mu}\chf L_{\mu},
$$
and hence (in $K[\Lambda^{+}(m)]$) \begin{equation}\label{equa53}
\pi\bigl(\chf\Delta_{\lam}\bigr)=\sum_{\mu\in\Lambda^{+}(m)}d_{\lam,\mu}\pi\bigl(\chf L_{\mu}\bigr).
\end{equation}
By the results in \cite{BLM} and \cite{Du}, $S_q(m)$ is an
epimorphic image of the quantum algebra $U_{K}(\mathfrak{gl}_m)$,
hence each irreducible $S_q(m)$-module $L_{\mu}$ is a highest weight
module over the quantum algebra $U_{K}(\mathfrak{gl}_m)$. It follows
easily that the elements in the set $\bigl\{\pi\bigl(\chf
L_{\mu}\bigr)\bigr\}$ are $K$-linear independent. Therefore, we can
compare the coefficients of the equalities (\ref{equa52}) and
(\ref{equa53}) and deduce that (in $K$)
$$
m_{\lam,\mu}=d_{\lam,\mu}\biggl(\frac{\sqrt{f_{(\lam,\lam)}(q)}}{\sqrt{f_{(\mu,\mu)}(q)}}
+d_{\lam,\mu}\biggr)/2.
$$
If $\mu$ is $e$-restricted, we then apply Corollary \ref{cor321} and this proves Theorem \ref{mainthm}.

\bigskip\bigskip\bigskip

\centerline{Acknowledgement}
\bigskip

\thanks{The research was supported by Alexander von Humboldt Foundation and partly by Scientific Research Foundation
for the Returned Overseas Chinese Scholars, State Education Ministry, the Basic Research Foundation of
BIT and the National Natural Science Foundation of China (Project 10771014). The author thanks the hospitality of the Mathematical Institute at the University of Cologne during his visit. He also thanks Professor Peter Littelmann for many helpful discussions.}

\end{document}